\documentclass{amsart}

\usepackage{lineno}
\usepackage{amssymb}
\usepackage{amsmath}
\usepackage{amsthm}
\usepackage{mathrsfs}
\usepackage{bbm}
\usepackage{cite}
\usepackage{color}
\usepackage{enumitem}
\usepackage{pgfplots}
\usepackage{tabularx}
\usepackage{booktabs}
\usepackage{array}
\usepackage{changes}

\usepackage{tikz-cd}

\usepackage{mathtools}
\usepackage{extarrows}
\usepackage{setspace}

\usepackage[colorlinks,hypertexnames=false]{hyperref}
\usepackage[msc-links]{amsrefs}

\hypersetup{
	colorlinks=true,
	linkcolor=black,
	filecolor=black,
	urlcolor=black,
	citecolor=black,
}

\makeatletter
\newcommand{\thickhline}{%
	\noalign {\ifnum 0=`}\fi \hrule height 1pt
	\futurelet \reserved@a \@xhline
}
\newcolumntype{"}{@{\hskip\tabcolsep\vrule width 1pt\hskip\tabcolsep}}

\@namedef{subjclassname@2020}{\textup{2020} Mathematics Subject Classification}

\makeatother

\newtheorem{thm}{Theorem}[section]
\newtheorem{cor}[thm]{Corollary}
\newtheorem{defn}[thm]{Definition}
\newtheorem{lem}[thm]{Lemma}
\newtheorem{prop}[thm]{Proposition}
\newtheorem{exa}[thm]{Example}
\newtheorem{rmk}[thm]{Remark}

\newtheorem{ass}[thm]{Assumption}

\newcommand{\diag}{\mathop{\mathrm{diag}}}

\newcommand{\End}{\mathop{\mathrm{End}}}

\numberwithin{equation}{section}

\begin{document}

\title[Slice regular functions over slice-cones]{A representation formula for slice regular functions over slice-cones in several variables}
\author{Xinyuan Dou}
\email[X.~Dou]{douxy@mail.ustc.edu.cn}
\address{Department of Mathematics, University of Science and Technology of China, Hefei 230026, China}
\author{Guangbin Ren}
\email[G.~Ren]{rengb@ustc.edu.cn}
\address{Department of Mathematics, University of Science and Technology of China, Hefei 230026, China}
\author{Irene Sabadini}
\email[Irene Sabadini]{irene.sabadini@polimi.it}
\address{Dipartimento di Matematica, Politecnico di Milano, Via Bonardi, 9, 20133 Milano, Italy}
\date{\today}


\thanks{This work was supported by the NNSF of China (11771412).}

\begin{abstract}
	The aim of this paper is to extend the so called slice analysis to a general case in which the codomain is a real vector space of even dimension, i.e. is of the form $\mathbb{R}^{2n}$. We define a cone $\mathcal{W}_\mathcal{C}^d$ in $[\End(\mathbb{R}^{2n})]^d$ and we extend the slice-topology $\tau_s$ to this cone. Slice regular functions can be defined on open sets in $\left(\tau_s,\mathcal{W}_\mathcal{C}^d\right)$ and a number of results can be proved in this framework, among which a representation formula. This theory can be applied to some real algebras, called left slice complex  structure algebras. These algebras include quaternions, octonions, Clifford algebras and real alternative $*$-algebras but also left-alternative algebras and sedenions, thus providing brand new settings in slice analysis.
\end{abstract}

\maketitle

\section{Introduction}\label{sc-in}

Quaternionic and more in general Clifford analysis are a classical subject which started over a hundred years ago, when complex analysts were looking for function theories generalizing holomorphic functions in one variable to higher dimensional cases. Functions ``hyperholomorphic" in a suitable sense, with values in an algebra extending that one of complex numbers, emerged as possible theories alternative to the one of holomorphic functions in several complex variables.
The list of contributors is quite long and we refer the interested reader to the recent \cite{Scetranslation} and the references therein.

In more recent times, inspired by a work of Cullen \cite{Cullen1965001}, Gentili and Struppa introduced a definition of hyperholomorphicity  \cite{Gentili2007001} over the quaternions $\mathbb{H}$. The corresponding class of functions includes convergent power series in the quaternionic variable $q$
\begin{equation*}
	\sum_{m\in\mathbb{N}} q^m a_m,\qquad a_m\in\mathbb{H}.
\end{equation*}

The original definition of slice regularity in \cite{Gentili2007001} requires that for any imaginary unit $I$ in $\mathbb{H}$, the restriction of $f$ to the complex plane $\mathbb{C}_I$ is holomorphic. In our terminology, this analysis is called weak slice analysis. Weak slice analysis has been extended to Clifford algebras, see \cite{Colombo2009002}, and octonions \cite{Gentili2010001}. This function theory has been widely developed in the past decade especially because of various applications to operator theory (see \cite{Colombo2011001B,Alpay2016001B} and the references therein), geometry (see e.g. \cite{Altavilla,Gentili2012001}), geometric function theory (see e.g. \cite{Ren2017002, Ren2017001, Ren2015001}) and since we have no pretense of completeness, we refer the reader to the references in the aformentioned books and in  \cite{bookentire,bookapprox,Gentili2013001B}.

Another definition of slice regularity, see \cite{Ghiloni2011001}, is given for functions with values in a real alternative $*$-algebra $A$, and defined on a domain $\Omega$ in the quadratic cone of $A$, and it is based on the fact that these functions have a holomorphic stem function. This class of functions, according to the open set where the functions are defined, may be different from the one previously discussed. In this paper these functions are called strong slice regular and strong slice analysis is the corresponding theory. This analysis can be extended to the case of several variables in the case of quaternions, Clifford algebras, octonions and real alternative $*$-algebras, see \cite{Colombo2012002},  \cite{Ghiloni2012001},  \cite{Ren2020001}, \cite{Ghiloni2020001}, respectively.

In the recent paper \cite{Dou2020001} a new topology $\tau_s$, called slice-topology, is defined on $\mathbb{H}$. This topology is finer than the Euclidean topology and it is used to further extend the class of weak slice regular functions which can be defined on open sets in $\tau_s$.
The slice-topology can be extended to the several variables case by using an idea firstly introduced in
\cite{Ren2020001} in the context of octonions. In that paper, the authors consider a cone
\begin{equation*}
	\mathbb{O}_s^d:=\bigcup_{I\in\mathcal{S}_\mathbb{O}}\mathbb{C}_I^d\ \left(\subset\mathbb{O}^d\right),\qquad d\in\mathbb{N}_+,
\end{equation*}
where
\begin{equation}\label{eq-so}
	\mathbb{S}_\mathbb{O}=\{I\in\mathbb{O}:I^2=-1\}
\end{equation}
is the set of imaginary units in octonions and $\mathbb{C}_I^d=(\mathbb{R}+\mathbb{R}I)^d$ is a $d$-dimensional complex space.

  Then, in \cite{Dou2020003}, the slice-topology is extended to a cone $\mathbb{O}_s^d$ in $\mathbb{O}^d$, so one can define weak slice regular functions in several octonionic variables. Following this approach, in this paper we extend the notion of weak slice regularity to a completely new and rather general setting. The main novelties are basically two: the first one is that we consider functions with values in a finite dimensional real vector space of even dimension, namely - up to isomorphisms - $\mathbb R^{2n}$. According to this choice of the codomain, the domain of the functions will be a set  $\mathcal{W}_\mathcal{C}^d$ in $\left[\End(\mathbb{R}^{2n})\right]^d$ that we call weak slice-cone, and this is the second novelty.
  The cone $\mathcal{W}_\mathcal{C}^d$ will be equipped with the slice-topology.

  Thus the weak slice regular functions studied in this work are $\mathbb{R}^{2n}$-valued and defined on open sets in the topological space $(\mathcal{W}_\mathcal{C}^d,\tau_s)$. For these functions, we prove various results among which a representation formula which generalizes the one in \cite{Colombo2009003, Colombo2009001, Dou2020001, Dou2020003}. A key tool for its proof is the Moore-Penrose inverse.

This new and general approach can be applied to a wide set of even-dimensional real algebras which are called left slice complex structure algebras, LSCS algebras for short, which include quaternions, octonions, Clifford algebras, real alternative $*$-algebras, some left alternative algebras, sedenions. Thus, not only we obtain the weak slice analysis covering all the previous known cases, but we also provide completely new settings.

Now, we describe the structure of the paper. In Section \ref{sc-pdb}, we recall some useful notions in linear algebra including the Moore-Penrose inverse and complex structures. In Section \ref{sc-wsc}, we define a weak slice-cone $\mathcal{W}_\mathcal{C}^d$ in $\left[\End(\mathbb{R}^{2n})\right]^d$ and we extend the slice-topology $\tau_s$ on it, also proving some of its properties. In Section \ref{sc-wsr}, weak slice regular functions are defined on open sets in the slice-topology $\tau_s$. We prove a splitting lemma, an identity principle and the maximum modulus principle for these functions. In the case of the latter result, we discuss an issue coming from the definition of a norm in $\mathbb{R}^{2n}$. Section \ref{sc-pf} is devoted to an extension lemma as well as some other useful results. In Section \ref{sc-prf}, we prove a representation formula for weak slice regular functions, we define path-slice functions on subsets of $\mathcal{W}_{\mathcal{C}}^d$ and we show that weak slice regular functions are path-slice. In Section \ref{sc-ws}, we define a kind of real algebras, called LSCS algebras, and we apply the above result to this framework.  In Section \ref{sc-laa}, we state some results for some real left alternative algebras, and compare our theory with classical slice analysis in the case of Clifford algebras and alternative $*$-algebras. In Section \ref{sc-sc}, we provide an example of this new slice analysis in the case of sedenions.

\section{Preliminaries}\label{sc-pdb}

In this section, we briefly recall some notions which will be useful in the sequel, including transformation matrices, the Moore-Penrose inverse and complex structures on $\mathbb{R}^{2n}$, $n\in\mathbb{N}_+=\{1,2,...,k,...\}$.

\subsection{Transformation matrix}
As it is well known in linear algebra, every real linear map $T:\mathbb{R}^\ell\rightarrow\mathbb{R}^k$, once that two bases of $\mathbb{R}^\ell$ and $\mathbb{R}^k$ are fixed, can regarded as a $k\times\ell$ real matrix $\Im(T)$. The matrix $\Im(T)$ is called the transformation matrix of $T$ and in this subsection we recall some of its properties.

Denote by $\End\left(\mathbb{R}^{2n}\right)$ the set of real linear maps from $\mathbb{R}^{2n}$ to $\mathbb{R}^{2n}$, which turns out to be a real associative algebra. Then, there exists an isomorphism between the real associative algebras $\End\left(\mathbb{R}^{m}\right)$ and $\mathbb{R}^{2n\times 2n}$ given by
\begin{equation*}
	\begin{split}
		\Im:\quad\End\left(\mathbb{R}^{2n}\right)\ &\xlongrightarrow[\hskip1cm]{}\ \mathbb{R}^{2n\times 2n}
		\\ T\quad &\shortmid\!\xlongrightarrow[\hskip1cm]{}\quad \Im(T),
	\end{split}
\end{equation*}
where $\Im(T)\in\mathbb{R}^{2n\times 2n}$ is the unique matrix such that for each $v=(v_1,...,v_{2n})^T\in\mathbb{R}^{2n}$,
\begin{equation*}
	T(v)=\Im(T)v.
\end{equation*}

Let $k,\ell\in\mathbb{N}_+$ be arbitrary but fixed. One can define an isomorphism, still denoted by $\Im$, between the real vector spaces $\End\left(\mathbb{R}^{2n}\right)^{k\times\ell}$ and $\mathbb{R}^{2nk\times 2n\ell}$, by setting
\begin{equation*}
	\begin{split}
		\Im:\quad \End\left(\mathbb{R}^{2n}\right)^{k\times\ell}\quad &\xlongrightarrow[\hskip1cm]{}\qquad\left(\mathbb{R}^{2n\times 2n}\right)^{k\times\ell}=\mathbb{R}^{2nk\times 2n\ell}
		\\E=\begin{pmatrix}
			E_{1,1}&\cdots&E_{1,\ell}
			\\\vdots& &\vdots
			\\E_{k,1}&\cdots&E_{k,\ell}
		\end{pmatrix}\ &\shortmid\!\xlongrightarrow[\hskip1cm]{}\ \begin{pmatrix}
			\Im(E_{1,1})&\cdots&\Im(E_{1,\ell})
			\\\vdots& &\vdots
			\\\Im(E_{k,1})&\cdots&\Im(E_{k,\ell})
		\end{pmatrix},
	\end{split}
\end{equation*}
where $E_{\imath,\jmath}\in\mathbb{R}^{2n\times 2n}$.

We also define an operator $E^*\in\End\left(\mathbb{R}^{2n}\right)^{\ell\times k}$ by
\begin{equation*}
	E^*:=\Im^{-1}\left([\Im(E)]^T\right).
\end{equation*}

Let $E\in\End\left(\mathbb{R}^{2n}\right)^{k\times\ell}$ and $F\in\mathbb{R}^{2nk\times 2n\ell}$. To simplify the notation we write $\Im_E$ instead of $\Im(E)$, and $\Im^{-1}_F$ instead of $\Im^{-1}(F)$. We finally recall the following result:

\begin{prop}
	Let $k,\ell,m\in\mathbb{N}_+$, $S,T\in\End\left(\mathbb{R}^{2n}\right)^{k\times\ell}$ and $U\in\End\left(\mathbb{R}^{2n}\right)^{\ell\times m}$. Then the following statements hold.
	\begin{enumerate}[label=(\roman*)]
		\item $\Im_{S+T}=\Im_S+\Im_T$
		\item $\Im_{SU}=\Im_S \Im_U$.
		\item $(SU)^*=U^*S^*$
		\item If $m=\ell$ and $U$ is invertible, then $\Im_{U^{-1}}=(\Im_{U})^{-1}$.
	\end{enumerate}
\end{prop}

\subsection{The Moore-Penrose inverse}

In this subsection, we revise the definition of Moore-Penrose inverse and some of its properties. It generalizes the notion of inverse of a matrix and we will use it to give a new representation formula in the Section \ref{sc-prf}.

For any $k\times\ell$ real matrix $M$, $k,\ell\in\mathbb{N}_+$, there is a unique matrix $M^+\in\mathbb{R}^{\ell\times k}$ (called the Moore–Penrose inverse of $M$, and denoted by $M^+$) that satisfies the Moore-Penrose conditions:

\begin{enumerate}[label=(\roman*)]
	\item $MM^+M=M$.
	\item $M^+MM^+=M^+$.
	\item $(MM^+)^T=MM^+$.
	\item $(M^+M)^T=M^+M$.
\end{enumerate}

We recall some useful properties of the Moore-Penrose inverse.
\begin{prop}\label{pr-ekmp}
	Let $\ell,k\in\mathbb{N}_+$, $M\in\mathbb{R}^{\ell\times\ell}$, $N\in\mathbb{R}^{\ell\times k}$ and $P\in\mathbb{R}^{k\times k}$.
	\begin{enumerate}[label=(\roman*)]
		\item If $M$ is invertible, then $M^{-1}=M^+$.
		\item If $M,Q$ are unitary, then $(MNQ)^+=Q^{-1}N^+M^{-1}$.
	\end{enumerate}
\end{prop}

Let $J\in\End\left(\mathbb{R}^{2n}\right)^{k\times\ell}$, then the Moore-Penrose inverse of $J$ can be written as
\begin{equation}\label{eq-jil}
	J^+:=\Im^{-1}\left((\Im_J)^+\right).
\end{equation}

Since $\Im^{-1}$ is an isomorphism, $J^+$ shares the above properties listed for $(\Im_J)^+$.
\begin{prop}\label{pr-jmp}
	Let $J\in\End\left(\mathbb{R}^{2n}\right)^{k\times\ell}$. Then $J^+$ is the unique matrix in $\End\left(\mathbb{R}^{2n}\right)^{\ell\times k}$ that satisfies the Moore-Penrose conditions:
	\begin{enumerate}[label=(\roman*)]
		\item\label{it-jmp1} $JJ^+J=J$.
		\item $J^+JJ^+=J^+$.
		\item $(JJ^+)^*=JJ^+$.
		\item $(J^+J)^*=J^+J$.
	\end{enumerate}
\end{prop}

It is clear that $U$ is unitary if and only if $\Im(U)$ is unitary. Then by Proposition \ref{pr-ekmp}, the following result holds.

\begin{prop}\label{pr-ekm}
	Let $\ell,k\in\mathbb{N}_+$, $I\in\End\left(\mathbb{R}^{2n}\right)^{\ell\times\ell}$, $J\in\End\left(\mathbb{R}^{2n}\right)^{\ell\times k}$ and $K\in\mathbb{R}^{k\times k}$.
	\begin{enumerate}[label=(\roman*)]
		\item\label{it-ekm1} If $I$ is invertible, then $I^{-1}=I^+$.
		\item\label{it-ekm2} If $I,K$ are unitary, then $(IJK)^+=K^{-1}J^+I^{-1}$.
	\end{enumerate}
\end{prop}

\begin{rmk}
As a motivation for this discussion on the Moore-Penrose transform we mention an example in slice octonionic analysis. We recall that $f:\mathbb{O}\rightarrow\mathbb{O}$ is called slice function if there is
$F:\mathbb{C}\rightarrow\mathbb{O}^{2\times 1}$, called the stem function of $f$, such that
\begin{equation*}
	f(x+yI)=(1,I)F(x+yi),\qquad\forall\ I\in\mathcal{S}_\mathbb{O},
\end{equation*}
where $\mathbb{S}_\mathbb{O}$ is defined by \eqref{eq-so}.
We can use the Moore-Penrose inverse to compute a solution $F$ to  the equations
\begin{equation*}
	f(x+yI)=(1,I)F(x+yi),\qquad I\in\{J_1,...,J_k\},
\end{equation*}
where $J_1,...,J_k\in\mathcal{S}_\mathbb{O}$. One of the solutions is
\begin{equation*}
	F(x+yi)=\begin{pmatrix}
		1&L_{J_1}\\\vdots&\vdots\\1&L_{J_k}
	\end{pmatrix}^+\begin{pmatrix}
		f(x+yJ_1)\\\vdots\\f(x+yJ_k)
	\end{pmatrix},
\end{equation*}
where $(\cdot)^+$ is the Moore-Penrose inverse defined by \eqref{eq-jil}.

\end{rmk}

\subsection{Complex structure}

We denote by $\mathfrak{C}_n$ the set of complex structures on $\mathbb{R}^{2n}$, i.e.
\begin{equation*}
	\mathfrak{C}_n:=\left\{T\in\End\left(\mathbb{R}^{2n}\right)\ :\ T^2=-1\right\},
\end{equation*}
where, for simplicity, we denote the identity map $id_{\mathbb{R}^{2n}}$ on $\mathbb{R}^{2n}$ by $1$. In this section, we recall some facts about complex structures.

\begin{defn}
For any $I\in\mathfrak{C}_n$, the set $\{\xi_1,...,\xi_n\}\subset\mathbb{R}^{2n}$ is called an $I$-basis of $\mathbb{R}^{2n}$ if
	\begin{equation*}
		\{\xi_1,...,\xi_n,I(\xi_1),...,I(\xi_n)\}
	\end{equation*}
	is a basis of $\mathbb{R}^{2n}$ as a real vector space.
\end{defn}

For any $I\in\mathfrak{C}_n$, let us choose a fixed $I$-basis of $\mathbb{R}^{2n}$ which is denoted by
\begin{equation}\label{eq-ti}
	\theta^I:=\{\theta^I_1,...,\theta^I_n\},
\end{equation}
and let us consider the ${2^n\times 2^n}$ real matrix $D_I$ given by
\begin{equation}\label{eq-dib} D_I:=\begin{pmatrix}\theta_1^I&\cdots&\theta_n^I&I\theta_1^I&\cdots&I\theta_n^I\end{pmatrix}.
\end{equation}

\begin{prop}
	The complex structures on $\mathbb{R}^{2n}$ are the operators in $\End\left(\mathbb{R}^{2n}\right)$ similar to \begin{equation}\label{eq-min}
		\mathbb{J}_{2n}=\begin{pmatrix*}
			& -\mathbb{I}_{n\times n}\\ \mathbb{I}_{n\times n}&
		\end{pmatrix*},
	\end{equation}
	i.e.
	\begin{equation*}
		\mathfrak{C}_n=\left\{\Im^{-1}\left(D\mathbb{
J}_{2n} D^{-1}\right):D\in GL_{2n}(\mathbb{R})\right\},
	\end{equation*}
	where $GL_{2n}(\mathbb{R})$ is the general linear group of $2n\times 2n$ matrices over $\mathbb{R}$.
	
	Moreover, for each $I\in\mathfrak{C}_n$, we have $D_I\in GL_{2n}(\mathbb{R})$ and
	\begin{equation}\label{eq-iii1}
		I=\Im^{-1}\left(D_I\mathbb{J}_{2n} D_I^{-1}\right).
	\end{equation}
\end{prop}

\section{Weak slice-cone}\label{sc-wsc}

In this section, we define a cone, which we call weak slice-cone $\mathcal{W}_\mathcal{C}^d$ of $\mathbb{R}^{2n}$. It is a subset of $\left[\End\left(\mathbb{R}^{2n}\right)\right]^d$ with a slice structure and plays the same role as the quadratic cone in a real alternative $*$-algebra. Then, we generalize the slice-topology in \cite{Dou2020001} to the weak slice-cone $\mathcal{W}_\mathcal{C}^d$ and we prove some properties  similar to those of the slice-topology in \cite{Dou2020001}.

\begin{defn}
	Let $\mathcal{C}\subset\mathfrak{C}_n$. We call
	\begin{equation*}
		\mathcal{W}_\mathcal{C}^d:=\bigcup_{I\in\mathcal{C}}\mathbb{C}_I^d
	\end{equation*}
	the $d$-dimensional weak slice-cone of $\mathcal{C}$, where
	\begin{equation*}
		\mathbb{C}_I^d:=\left(\mathbb{R}+\mathbb{R}I\right)^d.
	\end{equation*}
\end{defn}
By construction, we have
$
\mathbb{C}_I^d \subset \left[\End\left(\mathbb{R}^{2n}\right)\right]^d
$
for any $I\in\mathcal{C}\subset\mathfrak{C}_n$ so that $\mathcal{W}_\mathcal{C}^d\subset \left[\End\left(\mathbb{R}^{2n}\right)\right]^d$.
\begin{defn}
The set $\mathcal{C}\subset\mathfrak{C}_n$ is called symmetric, if
	\begin{equation*}
		\mathcal{C}=-\mathcal{C}:=\{-c\in\End\left(\mathbb{R}^{2n}\right):c\in\mathcal{C}\}.
	\end{equation*}
\end{defn}
We note that for each $\mathcal{C}\subset\mathfrak{C}_n$,
	\begin{equation*}
		\mathcal{W}_\mathcal{C}^d=\mathcal{W}_{\mathcal{C}^*}^d,
	\end{equation*}
	where $\mathcal{C}^*:=\mathcal{C}\cup(-\mathcal{C})$.
Thus, in the sequel we will assume the following:
\begin{ass}\label{as-cp}
	 Without loss of generality, we can assume that $\mathcal{C}$ is a non-empty symmetric subset of $\mathfrak{C}_n$. Thus, we can fix a subset $\mathcal{C}^+$ of $\mathcal{C}$ such that
	\begin{equation*}
		\mathcal{C}=\left(-\mathcal{C}^+\right)\bigsqcup\mathcal{C}^+.
	\end{equation*}
\end{ass}

Let $A$ be a (real) algebra. We denote by $L_a:A\rightarrow A$ the left-multiplication corresponding to $a\in A$, that is
\begin{equation*}
	L_a (b)=ab,\qquad\forall\ b\in A.
\end{equation*}

\begin{exa}
	We consider the case of the non-commutative algebra of real quaternions $\mathbb{H}$. As a real vector space it is isomorphic to $\mathbb{R}^4$. It is easy to check that the set
	\begin{equation*} \mathcal{C}_{\mathbb H}:=\left\{L_q\in\End\left(\mathbb{R}^4\right)\ :\ q\in\mathbb{H},\ q^2=-1\right\}
	\end{equation*}
	is a non-empty symmetric subset of $\mathfrak{C}_2=\mathfrak{C}(\mathbb{R}^{4})$. Writing $J$ instead of $L_J$, with a slight abuse of notation, we have that the set
	\begin{equation*}
		(\mathbb{H}\cong)\ L (\mathbb{H}):=\left\{L_a\ :\ a\in\mathbb{H}\right\}=\bigcup_{J\in\mathcal{C}_{\mathbb{H}}}\mathbb{C}_J\ \big(=\mathcal{W}_{\mathcal{C}_{\mathbb{H}}}^1\big)
	\end{equation*}
	is a $1$-dimensional weak slice-cone of $\mathcal{C}_{\mathbb{H}}$.
\end{exa}

Denote by $\tau\left(\left[\End\left(\mathbb{R}^{2n}\right)\right]^d\right)$ the Euclidean topology on the real vector space $\left[\End\left(\mathbb{R}^{2n}\right)\right]^d$. For each subset $U\subset\left[\End\left(\mathbb{R}^{2n}\right)\right]^d$, denote by $\tau(U)$ the subspace topology induced by $\tau\left(\left[\End\left(\mathbb{R}^{2n}\right)\right]^d\right)$. It is immediate that these two topologies  coincide on $\mathbb{C}_I^d$ for any $I\in\mathcal{C}\subseteq \mathfrak{C}_{n}$. We will refer to the subspace topology $\tau$ on $U$ as to the Euclidean topology on $U$.
\begin{lem}\label{lm-tlm}
	The set
	\begin{equation*}
		\tau_s\left(\mathcal{W}_{\mathcal{C}}^d\right):=\left\{\Omega\subset \mathcal{W}_{\mathcal{C}}^d\ :\ \Omega_I\in \tau(\mathbb{C}_I^d),\ \forall\ I\in\mathcal{C}\right\},
	\end{equation*}
	is a topology on $\mathcal{W}_{\mathcal{C}}^d$, where
	\begin{equation*}
		\Omega_I:=\Omega\cap\mathbb{C}_I^d.
	\end{equation*}
\end{lem}

\begin{proof}
	(i) It is clear that $\varnothing,\mathcal{W}_{\mathcal{C}}^d\in\tau_s\left(\mathcal{W}_{\mathcal{C}}^d\right)$.
	
	(ii) Let $U_1,...,U_k\in\tau_s\left(\mathcal{W}_{\mathcal{C}}^d\right)$. Then for any $I\in\mathcal{C}$ and $\ell\in\{1,...,k\}$,
	\begin{equation*}
		\left(U_\ell\right)_I\in\tau(\mathbb{C}_I^d).
	\end{equation*}
	It follows that
	\begin{equation*}
		\left(\bigcap_{\ell=1}^k U_\ell\right)\cap\mathbb{C}_I^d=\bigcap_{\ell=1}^k\left(U_\ell\right)_I
	\end{equation*}
	is open in $\mathbb{C}_I^d$. Since the choice of $I$ is arbitrary, it follows by definition that
	\begin{equation*}
		\bigcap_{\ell=1}^k U_\ell\in\tau_s\left(\mathcal{W}_{\mathcal{C}}^d\right).
	\end{equation*}
	
	(iii) Similarly, for any $U_\lambda\in\tau_s\left(\mathcal{W}_{\mathcal{C}}^d\right)$, $\lambda\in\Lambda$,
	\begin{equation*}
		\bigcup_{\lambda\in\Lambda} U_\lambda\in\tau_s\left(\mathcal{W}_{\mathcal{C}}^d\right).
	\end{equation*}
\end{proof}

\begin{defn}
	We call $\tau_s\left(\mathcal{W}_{\mathcal{C}}^d\right)$ the slice-topology on $\mathcal{W}_{\mathcal{C}}^d$.
\end{defn}

{\bf Convention}: Let $\Omega\subset\mathcal{W}_{\mathcal{C}}^d$. Denote by $\tau_s(\Omega)$ the subspace topology induced by $\tau_s\left(\mathcal{W}_{\mathcal{C}}^d\right)$. Open sets, domains, connected sets and paths in $\tau_s(\Omega)$ are called slice-open sets, slice-domains, slice-connected sets and slice-paths, respectively.

\begin{prop}\label{pr-sli}
	Let $\ell\in\mathbb{N}_+$, $I,J\in\mathcal{C}$ with $I\neq \pm J$. Then
	\begin{equation*}
		\mathbb{C}_I^\ell\cap\mathbb{C}_J^\ell=\mathbb{R}^\ell.
	\end{equation*}
\end{prop}

\begin{proof}
	Note that $\mathbb{C}_I$ can be identified with a complex plane. The equation $z^2=-1$, $z\in\mathbb{C}_I$, only has two solutions $z=\pm I$.
	
	Suppose that $\lambda+\mu J\in\mathbb{C}_I$, for some $\lambda\in\mathbb{R}$ and $\mu\in\mathbb{R}\backslash\{0\}$. It is clear that
	\begin{equation*}
		J=\mu^{-1}[(\lambda+\mu J)-\lambda]\in\mathbb{C}_I.
	\end{equation*}
	Since $J^2=-1$, we have $J=\pm I$, which is a contradiction. Hence $\mathbb{C}_I\cap\mathbb{C}_J=\mathbb{R}$ and then $\mathbb{C}_I^\ell\cap\mathbb{C}_J^\ell=\mathbb{R}^\ell$.
\end{proof}

\begin{prop}\label{pr-slm}
	$\left(\mathcal{W}_{\mathcal{C}}^d,\tau_s\right)$ is a Hausdorff space and $\tau\left(\mathcal{W}_{\mathcal{C}}^d\right)\subset\tau_s\left(\mathcal{W}_{\mathcal{C}}^d\right)$.
\end{prop}

\begin{proof}
	Let $\Omega\in\tau\left(\mathcal{W}_{\mathcal{C}}^d\right)$. By definition, there is an open set $U\in\left[\End\left(\mathbb{R}^{2n}\right)\right]^d$ such that $U\cap\mathcal{W}_{\mathcal{C}}^d=\Omega$. It is clear that
	\begin{equation*}
		U\cap\mathbb{C}_I^d=\Omega_I\in\tau(\mathbb{C}_I^d),\qquad\forall\ I\in\mathcal{C}.
	\end{equation*}
	Hence $\Omega\in\tau_s\left(\mathcal{W}_{\mathcal{C}}^d\right)$. It implies that $\tau\left(\mathcal{W}_{\mathcal{C}}^d\right)\subset\tau_s\left(\mathcal{W}_{\mathcal{C}}^d\right)$. Since $\left(\left[\End\left(\mathbb{R}^{2n}\right)\right]^d,\tau\right)$ is Hausdorff, the subspace $\left(\mathcal{W}_{\mathcal{C}}^d,\tau\right)$ is Hausdorff, so is $\left(\mathcal{W}_{\mathcal{C}}^d,\tau_s\right)$.
\end{proof}

It is easy to check that
\begin{equation*}
	\tau\left(\mathbb{R}^d\right)=\tau_s\left(\mathbb{R}^d\right)
\end{equation*}
and
\begin{equation}\label{eq-tmd}
	\tau\left(\mathbb{C}_I^d\right)=\tau_s\left(\mathbb{C}^d_I\right),\qquad\forall\ I\in\mathcal{C}.
\end{equation}

\begin{defn}
	A subset $\Omega$ of $\mathcal{W}_{\mathcal{C}}^d$ is called real-connected, if $$\Omega_{\mathbb{R}}:=\Omega\cap\mathbb{R}^d$$ is connected in $\mathbb{R}^d$. In particular, when $\Omega\cap\mathbb{R}^d=\varnothing$, $\Omega$ is real-connected, since the empty set is connected, by assumption.
\end{defn}

\begin{prop}\label{pr-wso}
	Let $\Omega$ be a slice-open set in $\mathcal{W}_{\mathcal{C}}^d$ and $q\in\Omega$. Then there is a real-connected slice-domain $U\subset\Omega$ containing $q$.
\end{prop}

\begin{proof}
The proof is similar to the proof of analogous statements in \cite{Dou2020001,Dou2020003}. We repeat here the main arguments: if $q\in\mathbb{R}^d$, then denote by $D$ the connected component of $\Omega_{\mathbb{R}}$ containing $q$ in $\mathbb{R}^d$; otherwise, set $D:=\varnothing$. Let $U$ be the slice-connected component of $(\Omega\backslash\Omega_{\mathbb{R}})\cup D$ containing $q$. It is easy to check that $q\in U\subset\Omega$ and $U$ is a real-connected slice-domain.
\end{proof}

Now we describe slice-connectedness by means of slice-paths.

\begin{defn}
	A path $\gamma$ in $\left(\mathcal{W}_{\mathcal{C}}^d,\tau\right)$ is called on a slice, if $\gamma\subset\mathbb{C}_I^d$ for some $I\in\mathcal{C}$.
\end{defn}

\begin{prop}
	Each path in $\left(\mathcal{W}_{\mathcal{C}}^d,\tau\right)$ on a slice is a slice-path.
\end{prop}

\begin{proof}
	This proposition holds directly by \eqref{eq-tmd}.
\end{proof}

\begin{prop}\label{pr-wat}
	Let $\Omega$ be a real-connected slice-domain in $\mathcal{W}_{\mathcal{C}}^d$. Then the following statements hold:
	\begin{enumerate}[label=(\roman*)]
		
		\item\label{it-wat12} If $\Omega_\mathbb{R}=\varnothing$, then $\Omega\subset\mathbb{C}_I^d$ for some $I\in\mathcal{C}$.
		
		\item\label{it-wat22} If $\Omega_\mathbb{R}\neq\varnothing$, then for any $q\in \Omega$ and $x\in \Omega_\mathbb{R}$, there is a path on a slice from $q$ to $x$.
		
	\end{enumerate}
\end{prop}

\begin{proof}
	\ref{it-wat12}. If $\Omega_\mathbb{R}=\varnothing$, then (see Assumption \ref{as-cp} for the definition of $\mathcal{C}^+$)
	\begin{equation*}
		\Omega\subset\bigsqcup_{J\in\mathcal{C}^+}\left(\mathbb{C}_J^d\backslash\mathbb{R}^d\right).
	\end{equation*}
	Note that $\mathbb{C}_J^d\backslash\mathbb{R}^d$ is slice-open in $\mathcal{W}_{\mathcal{C}}^d$ for each $J\in\mathcal{C}$. Since $\Omega$ is slice-connected, $\Omega\subset\mathbb{C}_I^d$ for some $I\in\mathcal{C}$.

	\ref{it-wat22}. Let $q\in\Omega$ and $x\in\Omega_\mathbb{R}$. Then $q\in\mathbb{C}_I^d$ for some $I\in\mathcal C$. Since $\Omega$ is a slice-domain in $\mathcal{W}_{\mathcal{C}}^d$, it follows by definition that $\Omega_I$ is an open set in $\mathbb C_I^d$.
	Denote by $V$ the connected component of $\Omega_I$ containing $q$. Then $V$ is also open in $\mathbb C_I^d$. By definition, $\mathbb{C}_I^d\backslash\mathbb{R}^d$ and $\bigcup_{J\in\mathcal{C}\backslash\{\pm I\}}\left(\mathbb{C}_J^d\backslash\mathbb{R}^d\right)$ are slice-open.
	
	If $V_\mathbb{R}=\varnothing$, then
	\begin{equation*}
		V=\Omega\cap\left(\mathbb{C}_I^d\backslash\mathbb{R}^d\right)\qquad\mbox{and}\qquad
		\Omega\backslash V=\Omega\cap\left[\bigcup_{J\in\mathcal{C}\backslash\{\pm I\}}\left(\mathbb{C}_J^d\backslash\mathbb{R}^d\right)\right]
	\end{equation*}
	are slice-open. Since $\Omega$ is slice-connected and nonempty, it follows from
	\begin{equation*}
		\Omega=V\ \bigsqcup\ (\Omega\backslash V)
	\end{equation*}
	that $V=\Omega$. Therefore $\Omega_\mathbb{R}=V_\mathbb{R}=\varnothing$, which is a contradiction.
	
	Otherwise, $V_\mathbb{R}\neq\varnothing$. Fix $x_0\in V_\mathbb{R}$. Since $V$ is the connected component of $\Omega_I$ containing $q$, there is a path $\alpha$ in $V$ from $q$ to $x_0$. Because $\Omega$ is real-connected, we have a path $\beta$ in $\Omega_\mathbb{R}$ from $x_0$ to $x$. It is clear that $\alpha\beta$ is a path on a slice from $q$ to $x$.
\end{proof}

\begin{cor}\label{co-wat}
	Let $\Omega$ be a real-connected slice-domain in $\mathcal{W}_{\mathcal{C}}^d$. Then the following statements hold:
	\begin{enumerate}[label=(\roman*)]
		
		\item\label{it-wat1} $\Omega_I$ is a domain in $\mathbb{C}_I^d$ for each $I\in\mathcal{C}$.
		
		\item Let $p,q\in\Omega$. Then there are two paths $\gamma_1,\gamma_2$ in $\Omega$ such that each of them is on a slice, $\gamma_1(1)=\gamma_2(0)$, and $\gamma_1\gamma_2$ is a slice-path from $p$ to $q$.
		
	\end{enumerate}
\end{cor}

\begin{proof}
	This follows directly from Proposition \ref{pr-wat}.
\end{proof}

\begin{prop}
	The topological space $\left(\mathcal{W}_{\mathcal{C}}^d,\tau_s\right)$ is connected, locally path-connected and path-connected.
\end{prop}

\begin{proof}
	It  follows from Proposition \ref{pr-wso} and Corollary \ref{co-wat} (ii) that $\left(\mathcal{W}_{\mathcal{C}}^d,\tau_s\right)$ is locally path-connected. Since $\mathcal{W}_{\mathcal{C}}^d\cap\mathbb{C}_I^d=\mathbb{C}_I^d\supset\mathbb{R}^d$ for each $I\in\mathcal{C}$, $(\mathcal{W}_{\mathcal{C}}^d,\tau_s)$ is path-connected. It implies that $(\mathcal{W}_{\mathcal{C}}^d,\tau_s)$ is also connected.
\end{proof}

\section{Weak slice regular functions}\label{sc-wsr}

As we wrote in the Introduction, the slice regular functions defined in \cite{Gentili2007001} are called weak slice regular functions in this paper. In \cite{Dou2020001}, we generalized these functions to open sets in the slice-topology $\tau_s$ on $\mathbb{H}$. In this section, we define $\mathbb{R}^{2n}$-valued weak slice regular functions on open sets in the slice-topology $\tau_s(\mathcal{W}_\mathcal{C}^d)$, and we prove a splitting lemma and an identity principle. To consider functions with values in a real vector space may seem to be reductive, and in fact in Section \ref{sc-ws} we will show that the study can be generalized to the case of functions with values in a suitable algebra; but on the contrary, it shows that part of the theory, including some powerful tools like the representation formula, in fact depends only on the vector space structure of the set of values not on its algebra structure.

\begin{defn}
	Let $\Omega\in\tau_s(\mathcal{W}_\mathcal{C}^d)$. A function $f:\Omega\rightarrow\mathbb{R}^{2n}$ is called weak slice regular if and only if for each $I\in\mathcal{C}$, $f_I:=f|_{\Omega_I}$ is (left $I$-)holomorphic, i.e. $f_I$ is real differentiable and for each $\ell=1,2,...,d$,
	\begin{equation*}
		\frac{1}{2}\left(\frac{\partial}{\partial x_\ell}+I\frac{\partial}{\partial y_\ell}\right)f_I(x+yI)=0,\qquad\mbox{on}\qquad\Omega_I.
	\end{equation*}
\end{defn}

\begin{exa}
	The monomial
	\begin{equation*}
		(x_1+y_1 I,...,x_d+y_d I)\shortmid\!\longrightarrow (x_1+y_1 I)^{m_1}\cdots(x_d+y_d I)^{m_d} a,
	\end{equation*}
	is a weak slice regular function, where $x_1,...,x_d,y_1,...,y_d\in\mathbb{R}$, $a\in\mathbb{R}^{2n}$, $m_1,...,m_d\in\mathbb{N}$ and $I\in\mathcal{C}$.
\end{exa}

\begin{lem}\label{lm-wsp}
	(Splitting Lemma) Let $\Omega\in\tau_s\left(\mathcal{W}_\mathcal{C}^d\right)$. A function $f:\Omega\rightarrow\mathbb{R}^{2n}$ is weak slice regular if and only if for any $I\in\mathcal{C}$ and $I$-basis $\{\xi_1,....,\xi_n\}$, there are $n$ holomorphic functions $F_1,...,F_n:\Omega_I\rightarrow\mathbb{C}_I$, such that
	\begin{equation*}
		f_I=\sum_{\ell=1}^n(F_\ell\xi_\ell).
	\end{equation*}
\end{lem}

\begin{proof}
	(i) Suppose that $f$ is weak slice regular. Let $I\in\mathcal{C}$ and $\{\xi_1,....,\xi_n\}$ be an $I$-basis. Then
	\begin{equation}\label{eq-mr}
		\mathbb{R}^{2n}=\mathbb{C}_I\xi_1\oplus\cdots\oplus\mathbb{C}_I\xi_n.
	\end{equation}
	Then there are $n$ functions $F_1,...,F_n:\Omega_I\rightarrow\mathbb{C}_I$, such that
	\begin{equation*}
		f_I=\sum_{\ell=1}^n(F_\ell\xi_\ell).
	\end{equation*}
	Since the composition of maps is associative in $\End\left(\mathbb{R}^{2n}\right)$ and $f$ is weak slice regular, we have for any $\imath\in\{1,...,d\}$,
	\begin{equation*}
		\begin{split}
			0=&\frac{1}{2}\left(\frac{\partial}{\partial x_\imath}+I\frac{\partial}{\partial y_\imath}\right)f_I
			=\sum_{\ell=1}^n \frac{1}{2}\left(\frac{\partial}{\partial x_\imath}+I\frac{\partial}{\partial y_\imath}\right)(F_\ell\xi_\ell)
			\\=&\sum_{\ell=1}^n\left[\frac{1}{2}\left(\frac{\partial}{\partial x_\imath}+I\frac{\partial}{\partial y_\imath}\right)F_\ell\right]\xi_\ell .
		\end{split}
	\end{equation*}
	Note that
	\begin{equation*}
		\left[\frac{1}{2}\left(\frac{\partial}{\partial x_\imath}+I\frac{\partial}{\partial y_\imath}\right)F_\ell\right]\xi_\ell\in\mathbb{C}_I\xi_\ell,\qquad\ell=1,...,n .
	\end{equation*}
It follows from \eqref{eq-mr} that for any $\imath\in\{1,...,d\}$ and $\ell\in\{1,...,n\}$,
	\begin{equation}\label{eq-lf}
		\left[\frac{1}{2}\left(\frac{\partial}{\partial x_\imath}+I\frac{\partial}{\partial y_\imath}\right)F_\ell\right]\xi_\ell=0.
	\end{equation}
	Since $\mathbb{R}^{2n}$ is an $n$-dimensional $\mathbb{C}_I$ (complex) vector space, for any $a\in\mathbb{R}^{2n}\backslash\{0\}\cong \mathbb C_I^n\backslash\{0\}$ and $z\in\mathbb{C}_I$, we have $z(a)=0$ if and only if $z=0$. Thus for any $\imath\in \{1,...,d\}$ and $\ell\in\{1,...,n\}$, \eqref{eq-lf} yields
	\begin{equation*}
		\frac{1}{2}\left(\frac{\partial}{\partial x_\imath}+I\frac{\partial}{\partial y_\imath}\right)F_\ell=0.
	\end{equation*}
	Hence $F_1,...,F_n$ are holomorphic.
	
	(ii) Suppose that for any choice of $I\in\mathcal{C}$ and of an $I$-basis $\{\xi_1^I,....,\xi_n^I\}$, there are $n$ holomorphic functions $F_1,...,F_n:\Omega_I\rightarrow\mathbb{C}_I$, such that
	\begin{equation*}
		f_I=\sum_{\ell=1}^n(F_\ell\xi_\ell^I).
	\end{equation*}
	 Then for any $\imath\in \{1,...,d\}$,
	\begin{equation*}
		\begin{split}
			\frac{1}{2}\left(\frac{\partial}{\partial x_\imath}+I\frac{\partial}{\partial y_\imath}\right)f_I
			=\sum_{\ell=1}^n\left[\frac{1}{2}\left(\frac{\partial}{\partial x_\imath}+I\frac{\partial}{\partial y_\imath}\right)F_\ell\right]\xi_\ell^I=0,
		\end{split}
	\end{equation*}
and so $f$ is weak slice regular by definition.
\end{proof}

\begin{lem}\label{lm-wip}
	Let $\Omega$ be a real-connected slice-domain in $\mathcal{W}_\mathcal{C}^d$, and $f,g:\Omega\rightarrow\mathbb{R}^{2n}$ be weak slice regular. Then the following statements holds.
	\begin{enumerate}[label=(\roman*)]
		\item\label{lm-wip1} If $\Omega_\mathbb{R}\neq\varnothing$ and $f,g$ coincide on a non-empty open subset of $\Omega_\mathbb{R}$, then $f=g$ on $\Omega$.
		\item If $f,g$ coincide on a non-empty open subset of $\Omega_I$ for some $I\in\mathcal{C}$, then $f=g$ on $\Omega$.
	\end{enumerate}
\end{lem}

\begin{proof}
	
	(i) Suppose that $f,g$ coincide on a non-empty open subset $U$ of $\Omega_\mathbb{R}$. Let $p\in U$ and $I\in\mathcal{C}$. By Splitting Lemma \ref{lm-wsp}, $f_I,g_I$ have same Taylor series at $p$. Hence there is an open set $V$ in $\mathbb{C}_I^d$ such that $f=g$ on $V$. By Corollary \ref{co-wat} \ref{it-wat1}, $\Omega_I$ is a non-empty domain in $\mathbb{C}_I^d$. Therefore $f=g$ on $\Omega_I$. Since the choice of $I$ is arbitrary, $f=g$ on $\Omega=\cup_{I\in\mathcal{C}}\Omega_I$.
	
	(ii) Suppose that $f,g$ coincide on a non-empty open subset of $\Omega_I$ for some $I\in\mathcal{C}$. By the classical Identity Principle (in several complex variables), $f=g$ on $\Omega_I$. If $\Omega_{\mathbb{R}}=\varnothing$, then by Proposition \ref{pr-wat} \ref{it-wat12}, $\Omega=\Omega_I$ so that $f=g$ on $\Omega$. Otherwise, $\Omega_I\cap\mathbb{R}^d\neq\varnothing$, hence $f=g$ on a non-empty open set $\Omega_\mathbb{R}=\Omega_I\cap\mathbb{R}^d$ in $\Omega_\mathbb{R}$. Therefore $f=g$ on $\Omega$ by \ref{lm-wip1}.
\end{proof}

\begin{thm}\label{tm-ip}
	(Identity Principle) Let $\Omega$ be a slice-domain in $\mathcal{W}_\mathcal{C}^d$ and $f,g:\Omega\rightarrow\mathbb{R}^{2n}$ be weak slice regular. Then the following statements holds.
	\begin{enumerate}[label=(\roman*)]
		\item\label{it-ip1} If $f=g$ on a non-empty open subset $D$ of $\Omega_{\mathbb{R}}$, then $f=g$ on $\Omega$.
		\item\label{it-ip2} If $f=g$ on a non-empty open subset $D$ of $\Omega_I$ for some $I\in\mathcal{C}$, then $f=g$ on $\Omega$.
	\end{enumerate}
\end{thm}

\begin{proof}
	Let us consider the set
	\begin{equation*}
		E:=\{x\in\Omega:\exists\ V\in\tau_s(\Omega),\ \mbox{s.t.}\ x\in V\ \mbox{and}\ f=g\ \mbox{on}\ V\},
	\end{equation*}
	which is a slice-open set in $\Omega$, by its definition.
		
	According to Proposition \ref{pr-wso}, there is a real-connected slice-domain $U$ in $\mathcal{W}_\mathcal{C}^d$ such that $U\cap D\neq\varnothing$ and $U\subset\Omega$. Since $U$ is slice-open, $U_I$ is open in $\mathbb{C}_I^d$ and $U_\mathbb{R}$ is open in $\mathbb{R}^d$. This fact implies that $U\cap D$ is non-empty and open in $\mathbb{C}_I^d$ (by \ref{it-ip2}) or in $\mathbb{R}^d$ (by \ref{it-ip1}). It follows from Lemma \ref{lm-wip} that $f=g$ on $U$. Hence $U\subset E$ and $E$ is nonempty.

	Let now $q\in\Omega\backslash E$. By Proposition \ref{pr-wso}, there is a real-connected slice-domain $V$ with $q\in V\subset\Omega$. Since $E$ and $V$ are slice-open, so is $E\cap V$.
	We have two cases: if $E\cap V \neq\varnothing$, then $f=g$ on the non-empty slice-open $E\cap V$. By Lemma \ref{lm-wip}, $f=g$ on $V$. It implies that $q\in E$, which is a contradiction.
	
	Otherwise, $E\cap V =\varnothing$ and it follows from $E,V\subset\Omega$ that $ V\subset\Omega\backslash E$. Hence $q\in V$ is a slice-interior point in $\Omega\backslash E$. Therefore $\Omega\backslash E$ is slice-open so that $E$ is closed in $\Omega$. Since $\Omega$ is a slice-connected, we deduce that $E=\Omega$ and the thesis follows.
\end{proof}

\begin{prop}
	(Maximum Modulus Principle) Suppose $\Omega$ be a slice-domain in $\mathcal{W}_\mathcal{C}^d$ and $f:\Omega\rightarrow\mathbb{R}^{2n}$ is a weak slice regular function. Let $|w|:=(w^T w)^{\frac{1}{2}}$, for any $w\in\mathbb{R}^{2n}$. If for some $p\in\Omega$,
	\begin{equation}\label{eq-qo}
	\sup_{q\in\Omega}|f(q)|=|f(p)|
	\end{equation}
	then $f$ is constant, that is $f\equiv f(p)$.
\end{prop}

\begin{proof}
	Suppose that $p=(p_1,...,p_d)\subset\mathbb{C}_I^d$ and  $\overline{P_{\mathbb{C}_I^d}(p,r)}\subset\Omega_I$ for some $I\in\mathbb{S}$ and $r\in\mathbb{R}_+$, where $P_{\mathbb{C}_I^d}(p,r):=\{(z_1,...,z_d)\in\mathbb{C}_I^n:|z_\ell-p_\ell|<r\}$ is a polydisc in $\mathbb{C}_I^d$, and $\overline{P_{\mathbb{C}_I^d}(p,r)}$ is the closure of $P_{\mathbb{C}_I^d}(p,r)$. The function
	\begin{equation*}
	z_1\rightarrow f(z_1,q_2,...,q_d)
	\end{equation*}
	is holomorphic on the closed ball $\overline{B_{\mathbb{C}_I}(p_1,r)}$. By the Splitting Lemma \ref{lm-wsp} and the mean value theorem for holomorphic functions in one variable,
	\begin{equation}\label{eq-fp}
	f(p_1,p_2,...,p_d)=\frac{1}{2\pi}\int_0^{2\pi}f(p_1+re^{I\theta},p_2,...,p_d)d\theta.
	\end{equation}
	It implies that
	\begin{equation*}
	|f(p)|\le \max_{\theta\in[0,2\pi)}|f(p_1+re^{I\theta},p_2,...,p_d)|.
	\end{equation*}
	By \eqref{eq-qo},
	\begin{equation}\label{eq-fpm}
	|f(p)|= \max_{\theta\in[0,2\pi)}|f(p_1+re^{I\theta},p_2,...,p_d)|.
	\end{equation}
	Since the restriction of $f$ to $\Omega_I$ is continuous, it follows from \eqref{eq-fp} and \eqref{eq-fpm}, that there is $C\in\mathbb{R}^{2n}$ such that
	\begin{equation*}
	C=f(p_1+re^{I\theta},p_2,...,p_d),\qquad\forall\ \theta\in[0,2\pi).
	\end{equation*}
	By the identity principle of holomorphic functions in one variable
	\begin{equation*}
	C=f(z_1,p_2,...,p_d),\qquad\forall\ z_1\in \overline{B_{\mathbb{C}_I}(p_1,r)},
	\end{equation*}
	and hence $f(p)=C$.
	
	For any fixed $z_1\in B_{\mathbb{C}_I}(p_1,r)$, the function
	\begin{equation*}
	z_2\rightarrow f(z_1,z_2,q_3,...,q_d)
	\end{equation*}
	holomorphic on $\overline{B_{\mathbb{C}_I}(p_2,r)}$, again attains its maximum modulus at  $(z_1,z_2,p_3,...,p_d)$, the center of $\overline{B_{\mathbb{C}_I}(p_2,r)}$, and hence is constant on $\overline{B_{\mathbb{C}_I}(p_2,r)}$. Iterating this procedure we obtain that $f(z)=f(p)$ for all $z\in \overline{P_{\mathbb{C}_I^d}(p,r)}$. By the Identity Principle \ref{tm-ip}, $f(z)=f(p)$ for all $z\in\Omega$.
\end{proof}

\begin{rmk} The above proof basically follows from that one of several complex variables. However, that proof does not immediately apply to our case, since not every norm can be used. For example, to consider $\mathbb R^{2n}\cong \mathbb C_I^n$, for a fixed $I\in\mathcal{C}$ and with the norm described below, the classical approach is not suitable to prove the result. Specifically, we say that a norm $|\cdot|_I$ on $\mathbb{R}^{2n}$ is an $I$-complex norm if
	\begin{equation*}
		|z a|_I=|z|_{\mathbb{C}_I}|a|_I,\qquad\forall\ z=x+yI\in\mathbb{C}_I\mbox{ and }a\in\mathbb{R}^{2n},
	\end{equation*}	
	where
	\begin{equation*}
		|z|_{\mathbb{C}_I}:=\sqrt{x^2+y^2}.
	\end{equation*}
	Since $I$ is a complex structure on $\mathbb{R}^{2n}$, then
			\begin{equation*}
				\mathbb{R}^{2n}=(\mathbb{C}_I\theta^I_1,...,\mathbb{C}_I\theta^I_n)\cong\mathbb{C}^{n},
			\end{equation*}
		where $\{\theta_1^I,...,\theta_n^I\}$ is a $I$-basis as in \eqref{eq-ti}. It is clear there is an $I$-complex norm $|\cdot|_I$ on $\mathbb{R}_I$ defined by
		\begin{equation*}
			|a|_I=\left(\sum_{\ell=1}^n (|a_\ell^I|_{\mathbb{C}_I})^2\right)^{\frac{1}{2}}
		\end{equation*}
	for any $a=\sum_{\ell=1}^na_\ell^I\theta_\ell^I$, where $a_\ell^I\in\mathbb{C}_I$.
The classical maximum modulus principle in several complex analysis is proved using the above norm $|\cdot|_I$ in $\mathbb{C}_I^n$.

However, the norm $|\cdot|$ that we used in our proof is not an $I$-complex norm, for some $I\in\mathcal{C}$. In fact, let us take the simplest case, i.e. $n=1$ and let us consider $\vartheta=(1,1)$ and $\theta_1=(1,0)$. Then for each $a\in\mathbb{R}^{2n}$ there is $a_1^\vartheta,a_2^\vartheta\in\mathbb{R}$ such that
\begin{equation*}
	a=a_1^\vartheta \theta_1+a_2^\vartheta\vartheta.
\end{equation*}
Define a real linear operator $I:\mathbb{R}^2\rightarrow\mathbb{R}^2$ by
\begin{equation*}
	I(a)=-a_2^\vartheta \theta_1+a_1^\vartheta \vartheta.
\end{equation*}
Then $I^2=-id$ and so $I$ is a complex structure. Since
\begin{equation*}
	|(1,0)|=1\qquad\mbox{and}\qquad|I(1,0)|=|(1,1)|=
	\sqrt{2}
	\neq 1=|I|_{\mathbb{C}_I} |(1,0)|,
\end{equation*}
it follows that $|\cdot|$ is not an $I$-complex norm. Moreover, in general, one can prove that any other norm
$|\cdot|'$ on $\mathbb{R}^{2n}$ is not a $J$-complex norm, for some $J\in\mathcal{C}$.
\end{rmk}

\section{Extension lemma}\label{sc-pf}

In this section we prove an extension lemma for weak slice regular functions with values in $\mathbb R^{2n}$. As a byproduct, we shall obtain other results of independent interest.

For each $I\in\mathcal{C}$, we define an isomorphism $\Psi_i^I$ by
\begin{equation*}
	\begin{split}
		\Psi_i^I:\quad\mathbb{C}^d\quad &\xlongrightarrow[\hskip1cm]{}\quad \mathbb{C}_I^d,
		\\ x+yi\ &\shortmid\!\xlongrightarrow[\hskip1cm]{}\ x+yI.
	\end{split}
\end{equation*}
(We note that the same isomorphism will be used in Section \ref{sc-ws} where we will work in an algebra $A$ and $\mathcal C$ will be denoted by $\mathcal{S}_A$).
Let $\gamma:[0,1]\rightarrow\mathbb{C}^d$ and $I\in\mathcal{C}$. Define a corresponding path in $\mathbb{C}_I^d$ by
\begin{equation*}
	\gamma^{I}:=\Psi_i^{I}\circ\gamma.
\end{equation*}

We now introduce the following set:
\begin{equation*}
	\mathscr{P}(\mathbb{C}^d):=\{\gamma:[0,1]\rightarrow\mathbb{C}^d,\ \gamma\ \mbox{is a path s.t. }\gamma(0)\in\mathbb{R}^d\};
\end{equation*}
 for any $\Omega\subset\mathcal{W}_\mathcal{C}^d$ we define
\begin{equation*} \mathscr{P}\left(\mathbb{C}^d,\Omega\right):=\left\{\delta\in\mathscr{P}\left(\mathbb{C}^d\right):\exists\ I\in\mathcal{C},\mbox{ s.t. } Ran(\delta^{I})\subset\Omega\right\},
\end{equation*}
and, finally, for an arbitrary, but fixed $\gamma\in\mathscr{P}\left(\mathbb{C}^d\right)$ we define
\begin{equation*}
	\mathcal{C}(\gamma,\Omega):=\left\{I\in\mathcal{C}: Ran(\gamma^{I})\subset\Omega\right\},
\end{equation*}
where $Ran(\cdot)$ is the image of a map.

\begin{lem}\label{lm-wlo}
	Let $\Omega\subset\mathcal{W}_\mathcal{C}^d$, $\gamma\in\mathscr{P}\left(\mathbb{C}^d,\Omega\right)$ and $J=(J_1,...,J_k)\in\left[\mathcal{C}(\gamma,\Omega)\right]^k$. Then there is a domain $U$ in $\mathbb{C}^d$ containing $\gamma([0,1])$ such that
	\begin{equation}\label{eq-pij}
		\Psi_i^{J_\ell}(U)\subset\Omega,\qquad\ell=1,...,k.
	\end{equation}
\end{lem}

\begin{proof}
	Let $\ell\in\{1,...,k\}$. Since $J\in\left[\mathcal{C}(\gamma,\Omega)\right]^k$, we have $\gamma^{J_\ell}\subset\Omega_{J_\ell}$. Hence for each $t\in[0,1]$, there is $r_{t,\ell}\in\mathbb{R}_+$ such that
	\begin{equation*}
		B_{J_\ell}(\gamma^{J_\ell}(t),r_{t,\ell})\subset\Omega,
	\end{equation*}
where $B_{J_\ell}(\gamma^{J_\ell}(t),r_{t,\ell})$ is the ball with center $\gamma^{J_\ell}(t)$ and radius $r_{t,\ell}$ in $\mathbb{C}_{J_\ell}^d$.
	Let us set
	\begin{equation*}
		r_t:=\min_{\ell=1,...,k}\{r_{t,\ell}\}.
	\end{equation*}
	Therefore
	\begin{equation*}
		U:=\bigcup_{t\in[0,1]}B(\gamma(t),r_t)
	\end{equation*}
	is a domain in $\mathbb{C}^d$ containing $\gamma([0,1])$ and satisfying \eqref{eq-pij}, since
	\begin{equation*}
		\begin{split}
			\Psi_i^{J_\ell}\left(U\right)=&\Psi_i^{J_\ell}\left(\bigcup_{t\in[0,1]}B(\gamma(t),r_t)\right)
			=\bigcup_{t\in[0,1]}B_{J_\ell}\left(\gamma^{J_\ell}(t),r_{t}\right)
			\\\subset&\bigcup_{t\in[0,1]}B_{J_\ell}\left(\gamma^{J_\ell}(t),r_{t,\ell}\right)\subset\Omega
		\end{split}
	\end{equation*}
	for all $\ell\in\{1,..,k\}$.
\end{proof}

Let $J=(J_1,...,J_k)\in\mathcal{C}^k$, $D_{J_\ell}$, be as in \eqref{eq-dib} for any $\ell=1,...,k$, and consider
\begin{equation*}
	D_J:=\begin{pmatrix}
		D_{J_1}\\&\ddots\\&&D_{J_k}
	\end{pmatrix},
	\qquad\diag(J):=\begin{pmatrix}
		J_1\\&\ddots\\&&J_k
	\end{pmatrix},
\end{equation*}
\begin{equation*}
	\zeta(J):=\begin{pmatrix}
		1&J_1\\\vdots &\vdots\\1&J_k
	\end{pmatrix}\qquad\mbox{and}\qquad\sigma:=\begin{pmatrix}
		&-1\\1
	\end{pmatrix}.
\end{equation*}

We call
\begin{equation}\label{eq-zpj}
	\zeta^+(J):=[\Im^{-1}_{D_J}\cdot\zeta(J)]^+ \Im^{-1}_{D_J}
\end{equation}
the $J$-slice inverse of $\zeta(J)$, where $[\Im^{-1}_{D_J}\cdot\zeta(J)]^+$ is the Moore-Penrose inverse of $\Im^{-1}_{D_J}\cdot\zeta(J)$ defined by \eqref{eq-jil}.

\begin{prop}\label{pr-i1i}
	Let $I\in\mathcal{C}$ and $J=(J_1,...,J_k)\in\mathcal{C}^k$. Then the following statements hold.
	\begin{enumerate}[label=(\roman*)]
		\item\label{it-i1i} $I(1,I)=-(1,I)\sigma$,
		\item\label{it-i1i2} $\diag(J)\zeta(J)=-\zeta(J)\sigma$,
		\item\label{it-i1i3} $\Im^{-1}_{D_J}\diag(J)(\Im^{-1}_{D_J})^{-1}$ is unitary,
		\item\label{it-i1i4} $I\left[(1,I)\zeta^+(J)\right]=\left[(1,I)\zeta^+(J)\right]\diag(J)$.
	\end{enumerate}
\end{prop}

\begin{proof}
	\ref{it-i1i} It is immediate to verify that
	\begin{equation*}
		I(1,I)=(I,-1)=(1,I)\begin{pmatrix}
			&-1\\1
		\end{pmatrix}=(1,I)\sigma.
	\end{equation*}

	\ref{it-i1i2} By \ref{it-i1i},
	\begin{equation*}
		\diag(J)\zeta(J)=\begin{pmatrix}
			J_1(1,J_1)\\\vdots\\J_k(1,J_k)
		\end{pmatrix}=\begin{pmatrix}
		-(1,J_1)\sigma\\\vdots\\-(1,J_k)\sigma
		\end{pmatrix}=-\zeta(J)\sigma.
	\end{equation*}

	\ref{it-i1i3} By \eqref{eq-iii1},
	\begin{equation*}
		\begin{split}
			\Im^{-1}_{D_J}\diag(J)(\Im^{-1}_{D_J})^{-1}=&\begin{pmatrix}
				\Im^{-1}_{D_{J_1}}J_1(\Im^{-1}_{D_{J_1}})^{-1}\\&\ddots\\&&\Im^{-1}_{D_{J_k}}J_k(\Im^{-1}_{D_{J_k}})^{-1}
			\end{pmatrix}
			\\=&
			\begin{pmatrix}
				\Im^{-1}_{\mathbb{J}_{2n}}\\&\ddots\\&&\Im^{-1}_{\mathbb{J}_{2n}}
			\end{pmatrix}.
		\end{split}
	\end{equation*}
	It is clear that $\Im^{-1}_{D_J}\diag(J)(\Im^{-1}_{D_J})^{-1}$ is unitary.
	
	\ref{it-i1i4}  Since $\sigma$ is unitary, it follows from Proposition \ref{pr-ekm} \ref{it-ekm2} that
	\begin{equation}\label{eq-si}
		-\sigma[\Im^{-1}_{D_J}\cdot\zeta(J)]^+=\sigma^{-1}[\Im^{-1}_{D_J}\cdot\zeta(J)]^+=[\Im^{-1}_{D_J}\cdot\zeta(J)\sigma]^+.
	\end{equation}
	By \ref{it-i1i3}, $\Im^{-1}_{D_J}\diag(J)(\Im^{-1}_{D_J})^{-1}$ is unitary. Again according to Proposition \ref{pr-ekm} \ref{it-ekm2},
	\begin{equation}\label{eq-sdj}
		\begin{split}
			\left[\Im^{-1}_{D_J}\diag(J)\zeta(J)\right]^+=&\left[\Im^{-1}_{D_J}\diag(J)(\Im^{-1}_{D_J})^{-1}\Im^{-1}_{D_J}\zeta(J)\right]^+
			\\=&\left[(\Im^{-1}_{D_J})\zeta(J)\right]^+\left[\Im^{-1}_{D_J}\diag(J)(\Im^{-1}_{D_J})^{-1}\right]^{-1}
			\\=&\left[(\Im^{-1}_{D_J})\zeta(J)\right]^+\Im^{-1}_{D_J}[-\diag(J)](\Im^{-1}_{D_J})^{-1}.
		\end{split}
	\end{equation}
We then deduce the following chain of equalities
	\begin{equation*}
		\begin{split}
			&I\left[(1,I)[\Im^{-1}_{D_J}\cdot\zeta(J)]^+ \Im^{-1}_{D_J}\right]=-(1,I)\sigma[\Im^{-1}_{D_J}\cdot\zeta(J)]^+ \Im^{-1}_{D_J}
			\\=&(1,I)[\Im^{-1}_{D_J}\cdot\zeta(J)\sigma]^+\Im^{-1}_{D_J}
			=-(1,I)[\Im^{-1}_{D_J}\diag(J)\zeta(J)]^+\Im^{-1}_{D_J}
			\\=&-(1,I)\left[(\Im^{-1}_{D_J})\zeta(J)\right]^+\Im^{-1}_{D_J}[-\diag(J)](\Im^{-1}_{D_J})^{-1}\Im^{-1}_{D_J}
			\\=&\left[(1,I)\left[(\Im^{-1}_{D_J})\zeta(J)\right]^+\Im^{-1}_{D_J}\right]\diag(J),
		\end{split}
	\end{equation*}
	where the first equality holds by \ref{it-i1i}, the second, third and fourth equalities follow from \eqref{eq-si}, \ref{it-i1i2}, and \eqref{eq-sdj}, respectively. We conclude that \ref{it-i1i4} holds.
\end{proof}

Let $J=(J_1,...,J_k)\in\mathcal{C}^k$, $\Omega\subset\mathcal{W}_\mathcal{C}^d$ and $\gamma\in\mathscr{P}(\mathbb{C}^d,\Omega)$. We define
\begin{equation*} \mathcal{C}_{ker}(J):=\left\{I\in\mathcal{C}:\ker(1,I)\supset\bigcap_{\ell=1}^k\ker(1,J_\ell)\right\},
\end{equation*}
and
\begin{equation*}
	\mathcal{C}(\Omega,\gamma,J):=\mathcal{C}(\Omega,\gamma)\cap\mathcal{C}_{ker}(J),
\end{equation*}
where $\ker(\cdot)$ is the kernel of a map, and $\ker(1,J_\ell)$ stands for $\ker((1,J_\ell))$.

\begin{prop}\label{pr-jjk}
	Let $J=(J_1,...,J_k)\in\mathcal{C}^k$. Then
	\begin{equation}\label{eq-il}
		Ran\left[\zeta^+(J)\zeta(J)-id_{(\mathbb{R}^{2n})^{2\times 1}}\right]\subset \ker[\zeta(J)]=\bigcap_{\ell=1}^k \ker(1,J_\ell).
	\end{equation}
	Moreover, for each $I\in\mathcal{C}_{ker}(J)$,
	\begin{equation}\label{eq-ij}
		(1,I)\zeta^+(J)\zeta(J)=(1,I).
	\end{equation}
\end{prop}

\begin{proof}
	(i) By Proposition \ref{pr-jmp} \ref{it-jmp1},
	\begin{equation*}
		[\Im^{-1}_{D_J}\cdot\zeta(J)]=[\Im^{-1}_{D_J}\cdot\zeta(J)][\Im^{-1}_{D_J}\cdot\zeta(J)]^+[\Im^{-1}_{D_J}\cdot\zeta(J)].
	\end{equation*}
	Since $\Im^{-1}_{D_J}$ is invertible, it follows from \eqref{eq-zpj} that
	\begin{equation*}
		\zeta(J)=\zeta(J)\zeta^+(J)\zeta(J).
	\end{equation*}
	Hence
	\begin{equation*}
		\zeta(J)\left[\zeta^+(J)\zeta(J)-id_{(\mathbb{R}^{2n})^{2\times 1}}\right]=0.
	\end{equation*}
	It is clear that \eqref{eq-il} holds by the above equality.
	
	(ii) Let $I\in\mathcal{C}(\Omega,\gamma,J)$. By definition and \eqref{eq-il},
	\begin{equation*}
		\ker(1,I)\supset\bigcap_{\ell=1}^k \ker(1,J_\ell)\supset Ran\left[\zeta^+(J)\zeta(J)-id_{(\mathbb{R}^{2n})^{2\times 1}}\right].
	\end{equation*}
	It implies that
	\begin{equation*}
		(1,I)\zeta^+(J)\zeta(J)-(1,I)=(1,I)\left[\zeta^+(J)\zeta(J)-id_{(\mathbb{R}^{2n})^{2\times 1}}\right]=0.
	\end{equation*}
	Hence equality \eqref{eq-ij} holds.
\end{proof}

\begin{lem}\label{lm-lum}
	Let $U\in\tau(\mathbb{C}^d)$, $I\in\mathcal{C}$ and $J=(J_1,...,J_k)\in\mathcal{C}^k$. If $g_\ell:\Psi_i^{J_\ell}(U)\rightarrow\mathbb{R}^{2n}$, $\ell=1,...,k$ are holomorphic, then the function $g[I]:\Psi_i^I(U)\rightarrow\mathbb{R}^{2n}$ defined by
	\begin{equation}\label{eq-gix}
		g[I](x+yI)=(1,I)
		\zeta^+(J)
		g(x+yJ),\qquad\forall\ x+yi\in U,
	\end{equation}
	where
	\begin{equation}\label{eq-gvk}
		g(x+yJ)=\begin{pmatrix}
			g_1(x+yJ_1)\\\vdots\\g_k(x+yJ_k)
		\end{pmatrix}
	\end{equation}
	is holomorphic.

	Moreover, if $U_\mathbb{R}:=U\cap\mathbb{R}^d\neq\varnothing$, $g_1=\cdots=g_k$ on $U_\mathbb{R}$ and $I\in\mathcal{C}_{ker}(J)$, then
	\begin{equation}\label{eq-gg}
		g[I]=g_1=\cdots=g_k\qquad\mbox{on}\qquad U_\mathbb{R}.
	\end{equation}
\end{lem}

\begin{proof}
	(i) By Proposition \ref{pr-i1i} \ref{it-i1i4}, for each $\ell\in\{1,...,d\}$ and $x+yi\in U$,
	\begin{equation*}
		\begin{split}
			&\frac{1}{2}\left(\frac{\partial}{\partial x_\ell}+I\frac{\partial}{\partial y_\ell}\right)g[I](x+yI)
			\\=&\frac{1}{2}\left(\frac{\partial}{\partial x_\ell}+I\frac{\partial}{\partial y_\ell}\right)(1,I)\zeta^+(J)g(x+yJ)
			\\=&(1,I)\zeta^+(J)\begin{pmatrix}
				\frac{1}{2}\left(\frac{\partial}{\partial x_\ell}+J_1\frac{\partial}{\partial y_\ell}\right)
				\\&\ddots
				\\&&\frac{1}{2}\left(\frac{\partial}{\partial x_\ell}+J_k\frac{\partial}{\partial y_\ell}\right)
			\end{pmatrix}\begin{pmatrix}
				g(x+yJ_1)\\\vdots\\g(x+yJ_k)
			\end{pmatrix}
			\\=&(1,I)\zeta^+(J)\begin{pmatrix}
				\frac{1}{2}\left(\frac{\partial}{\partial x_\ell}+J_1\frac{\partial}{\partial y_\ell}\right)g(x+yJ_1)\\\vdots\\\frac{1}{2}\left(\frac{\partial}{\partial x_\ell}+J_k\frac{\partial}{\partial y_\ell}\right)g(x+yJ_k)
			\end{pmatrix}=0.
		\end{split}
	\end{equation*}
	Hence $g[I]$ is holomorphic.
	
	(ii) Suppose that $U_\mathbb{R}\neq\varnothing$, $g_1=\cdots=g_k$ on $U_\mathbb{R}$ and $I\in\mathcal{C}_{ker}(J)$. Then
	\begin{equation*}
		\begin{pmatrix}
			g_1(x)\\\vdots\\ g_k(x)
		\end{pmatrix}=\begin{pmatrix}
			1&J_1\\\vdots&\vdots\\1&J_k
		\end{pmatrix}\begin{pmatrix}g_1(x)\\0\end{pmatrix}
		=\zeta(J)\begin{pmatrix}g_1(x)\\0\end{pmatrix}.
	\end{equation*}
	On the other hand, by \eqref{eq-ij}, we have
	\begin{equation*}
		(1,I)\zeta^+(J)\zeta(J)=(1,I).
	\end{equation*}
	Hence for each $x\in U_\mathbb{R}$,
	\begin{equation*}
		\begin{split}
			g[I](x)=&(1,I)\zeta^+(J)
			\begin{pmatrix}
				g_1(x)\\\vdots\\ g_k(x)
			\end{pmatrix}
			=(1,I)\zeta^+(J)\zeta(J)\begin{pmatrix}g_1(x)\\0\end{pmatrix}
			\\=&(1,I)\begin{pmatrix}g_1(x)\\0\end{pmatrix}=g_1(x).
		\end{split}
	\end{equation*}
	It follows that \eqref{eq-gg} holds.
\end{proof}

\section{Path-representation formula}\label{sc-prf}

In this section, we prove a weak path-representation formula for weak slice regular functions. We also define path-slice functions and show that weak slice regular functions are path-slice.

\begin{thm}\label{th-oi}
	(Path-representation Formula) Let $\Omega$ be a slice-open set in $\mathcal{W}_\mathcal{C}^d$, $\gamma\in\mathscr{P}(\mathbb{C}^d,\Omega)$, $J=(J_1,J_2,...,J_k)\in\left[\mathcal{C}(\gamma,\Omega)\right]^k$ and $I\in\mathcal{C}(\gamma,\Omega,J)$.
    If $f:\Omega\rightarrow\mathbb{R}^{2n}$ is weak slice regular, then
	\begin{equation}\label{eq-ki}
		f\circ\gamma^I=(1,I)\zeta^+(J)(f\circ\gamma^J),
	\end{equation}
	where
	\begin{equation}\label{eq-fgj}
		f\circ\gamma^J:=\begin{pmatrix}
			f\circ\gamma^{J_1}\\\vdots\\f\circ\gamma^{J_k}
		\end{pmatrix}.
	\end{equation}
\end{thm}

\begin{proof}
	By Lemma \ref{lm-wlo}, there is a domain $U$ in $\mathbb{C}^d$ containing $\gamma([0,1])$ such that
	\begin{equation*}
		\Psi_i^{K}(U)\subset\Omega,\qquad K=I,J_1,...J_k.
	\end{equation*}
	Since $f$ is weak slice regular, $f|_{\Psi_i^{J_\ell}(U)}$, $\ell=1,...,k$ are holomorphic.
	
	Note that $I\in\mathcal{C}_{ker}(J)$. By Lemma \ref{lm-lum}, there is a holomorphic function $g:\Psi_i^{I}(U)\rightarrow\mathbb{R}^{2n}$, such that for each $x+yi\in U$,
	\begin{equation*}
		g(x+yI)=(1,I)\zeta^+(J)f(x+yJ),
	\end{equation*}
	and
	\begin{equation*}
		g=f|_{\Psi_i^{J_1}(U)}=f=f_I,\qquad\mbox{on}\qquad U_\mathbb{R},
	\end{equation*}
	where
	\begin{equation*}
		f(x+yJ)=\begin{pmatrix}
			f(x+yJ_1)\\\vdots\\f(x+yJ_k)
		\end{pmatrix}.
	\end{equation*}
	Since $g,f_I$ are holomorphic on a domain $\Psi_i^{I}(U)$ in $\mathbb{C}_I^d$ and $g=f_I$ on $U_\mathbb{R}\subset\Psi_i^{I}(U)$, it follows from the Taylor series expansion and the identity principle in complex analysis that
	\begin{equation*}
		g=f_I,\qquad\mbox{on}\qquad \Psi_i^{I}(U).
	\end{equation*}

	Let $t\in[0,1]$ and consider
	\begin{equation*}
		\gamma(t)=x_t+y_t i,\qquad\mbox{for some }x_t,y_t\in\mathbb{R}^d.
	\end{equation*}
	Then
	\begin{equation*}
		\gamma^K(t)=x_t+y_t K,\qquad K=I,J_1,...,J_k.
	\end{equation*}
	It is clear that for each $t\in[0,1]$,
	\begin{equation*}
		\begin{split}
			&f\circ\gamma^I(t)=f(x_t+y_t I)=g(x_t+y_t I)
			\\=&(1,I)\zeta^+(J)\begin{pmatrix}
				f(x_t+y_tJ_1)\\ \vdots\\f(x_t+y_tJ_k)
			\end{pmatrix}
			=(1,I)\zeta^+(J)\begin{pmatrix}
				f\circ\gamma^{J_1}(t)\\ \vdots\\f\circ\gamma^{J_k}(t)
			\end{pmatrix}.
		\end{split}
	\end{equation*}
	It implies that \eqref{eq-ki} holds.
\end{proof}

\begin{defn}
	Let $\Omega\subset\mathcal{W}_\mathcal{C}^d$, $\gamma\in\mathscr{P}(\mathbb{C}^d,\Omega)$ and $J=(J_1,...,J_k)\in \left[\mathcal{C}(\Omega,\gamma)\right]^k$. We say that $J$ is a slice-solution of $\mathcal{C}(\Omega,\gamma)$ if
	\begin{equation*}
		\mathcal{C}(\Omega,\gamma)=\mathcal{C}(\Omega,\gamma,J).
	\end{equation*}
\end{defn}

\begin{exa}\label{ex-ij}
	Let $\Omega\subset\mathcal{W}_\mathcal{C}^d$ and $\gamma\in\mathscr{P}(\mathbb{C}^d,\Omega)$. If $\pm I\in\mathcal{C}(\Omega,\gamma)$, then $(I,-I)$ is a slice-solution of $\mathcal{C}(\gamma,\Omega)$. This is because
	\begin{equation*}
		\begin{pmatrix}
			1&I\\1&-I
		\end{pmatrix}^{-1}=
		\frac{1}{2}\begin{pmatrix}
			1&1\\-I&I
		\end{pmatrix},
	\end{equation*}
	and
	\begin{equation*}
		\ker(1,I)\cap\ker(1,-I)=\ker\left[\begin{pmatrix}
			1&I\\1&-I
		\end{pmatrix}\right]=0.
	\end{equation*}

	Similarly, if $J_1,J_2\in\mathcal{C}(\Omega,\gamma)$ with $J_1-J_2$ being invertible, then $(J_1,J_2)$ is a slice-solution of $\mathcal{C}(\gamma,\Omega)$ (cf. \eqref{eq-1ji} below).
\end{exa}

\begin{prop}\label{pr-lss}
	Let $\Omega\subset\mathcal{W}_\mathcal{C}^d$ and $\gamma\in\mathscr{P}(\mathbb{C}^d,\Omega)$. Then there is at least one slice-solution $J\in \left[\mathcal{C}(\Omega,\gamma)\right]^k$ of $\mathcal{C}(\Omega,\gamma)$ for some $k\in\mathbb{N}_+$.
\end{prop}

\begin{proof}
	Let $\ell\in\mathbb{N}_+$. Define
	\begin{equation*}
		m_{\ell}:=\min\left\{\dim_{\mathbb{R}}\left(\bigcap_{\jmath=1}^\ell \ker(1,K_\jmath)\right):K_\jmath\in\mathcal{C}(\Omega,\gamma)\right\},
	\end{equation*}
	and
	\begin{equation*}
		m_{0}:=\dim_{\mathbb{R}}\left(\bigcap_{K\in\mathcal{C}(\Omega,\gamma)} \ker(1,K_\jmath)\right).
	\end{equation*}
	It is clear that $m_0\le m_{\ell}$ for each $\ell\in\mathbb{N}_+$.
	
	Suppose that $I=(I_1,...,I_\ell)\in\left[\mathcal{C}(\Omega,\gamma)\right]^\ell$ such that
	\begin{equation*}
		m_{\ell}=\dim_{\mathbb{R}}\left(\bigcap_{\jmath=1}^\ell \ker(1,I_\jmath)\right).
	\end{equation*}
	If $m_\ell>m_0$, then there is $I_{\ell+1}\in\mathcal{C}(\Omega,\gamma)$ such that $\ker(1,I_{\ell+1})\nsupseteq\bigcap_{\jmath=1}^\ell \ker(1,I_\jmath)$. Therefore,
	\begin{equation*}
		m_{\ell+1}\le\dim_{\mathbb{R}}\left(\bigcap_{\jmath=1}^{\ell+1} \ker(1,I_\jmath)\right)
		\le\dim_{\mathbb{R}}\left(\bigcap_{\jmath=1}^\ell \ker(1,I_\jmath)\right)-1
		\le m_\ell-1.
	\end{equation*}
	It implies that there is $k\in\mathbb{N}_+$ such that $m_k=m_0$. By definition there is $J=(J_1,...,J_k)\in \left[\mathcal{C}(\Omega,\gamma)\right]^k$, such that $m_0=m_k=\dim_{\mathbb{R}}\left(\bigcap_{\jmath=1}^k \ker(1,J_\jmath)\right)$. It follows that for each $I\in\mathcal{C}(\Omega,\gamma)$,
	\begin{equation*}
		\bigcap_{\jmath=1}^K \ker(1,J_\jmath)=\bigcap_{K\in\mathcal{C}(\Omega,\gamma)}\ker(1,K_\jmath)\subset\ker(1,I).
	\end{equation*}
	By definition, for each $I\in\mathcal{C}(\Omega,\gamma)$, we have $I\in\mathcal{C}(\Omega,\gamma,J)$. Hence $\mathcal{C}(\Omega,\gamma,J)=\mathcal{C}(\Omega,\gamma)$ and $J$ is a slice-solution of $\mathcal{C}(\Omega,\gamma)$.
\end{proof}

Now we generalize the notion of path-slice functions to $\mathcal{W}_{\mathcal{C}}^d$. We also show that weak slice regular functions are weak path-slice.

\begin{defn}\label{df-wlos}
	Let $\Omega\subset\mathcal{W}_{\mathcal{C}}^d$. A function $f:\Omega\rightarrow\mathbb{R}^{2n}$ is called path-slice if there is a function $F:\mathscr{P}(\mathbb{C}^d,\Omega)\rightarrow (\mathbb{R}^{2n})^{2\times 1}$ such that
	\begin{equation}\label{eq-fcg}
		f\circ\gamma^{I}(1)=(1,I)F(\gamma),
	\end{equation}
	for any $\gamma\in\mathscr{P}(\mathbb{C}^d,\Omega)$ and $I\in\mathcal{C}(\gamma,\Omega)$, where $f\circ\gamma^J$ is defined by \eqref{eq-fgj}.
	
	We call $F$ a path-slice stem function of $f$.
\end{defn}

\begin{defn}
	Let $\Omega\subset\mathcal{W}_\mathcal{C}^d$. A function $f:\Omega\rightarrow\mathbb{R}^{2n}$ is called path-pseudoslice if for any $\gamma\in\mathscr{P}(\mathbb{C}^d,\Omega)$, $J\in\left[\mathcal{C}(\gamma,\Omega)\right]^k$ and $I\in\mathcal{C}(\gamma,\Omega,J)$,
	\begin{equation*}
		f\circ\gamma^I=(1,I)\zeta^+(J)(f\circ\gamma^J),
	\end{equation*}
	where $f\circ\gamma^J$ is defined by \eqref{eq-fgj}.
\end{defn}

\begin{defn}
	Let $\Omega\subset\mathcal{W}_\mathcal{C}^d$. A function $f:\Omega\rightarrow\mathbb{R}^{2n}$ is called path-solution slice, if for any $\gamma\in\mathscr{P}(\mathbb{C}^d,\Omega)$ and slice-solution $J\in\left[\mathcal{C}(\gamma,\Omega)\right]^k$ of $\mathcal{C}(\gamma,\Omega)$,
	\begin{equation}\label{eq-fcgi}
		f\circ\gamma^I=(1,I)\zeta^+(J)(f\circ\gamma^J),\qquad\forall\ I\in\mathcal{C}(\gamma,\Omega),
	\end{equation}
	where $f\circ\gamma^J$ is defined by \eqref{eq-fgj}.
\end{defn}

Let $\gamma\in\mathscr{P}(\mathbb{C}^d)$ and $t\in[0,1]$. Define the path $\gamma[t]:[0,1]\rightarrow\mathbb{C}^d$ by setting
\begin{equation*}
	\gamma[t](s)=\gamma(ts),\qquad\forall\ s\in[0,1].
\end{equation*}
It is clear that $\gamma[t]\in\mathscr{P}(\mathbb{C}^d)$.

\begin{prop}\label{pr-fsa}
	Let $\Omega\subset\mathcal{W}_\mathcal{C}^d$ and $f:\Omega\rightarrow\mathbb{R}^{2n}$. Then the following statements are equivalent.
	\begin{enumerate}[label=(\roman*)]
		\item\label{it-lo1} $f$ is path-slice
		\item\label{it-lo2} $f$ is path-pseudoslice,
		\item\label{it-lo3} $f$ is path-solution slice.
	\end{enumerate}
\end{prop}

\begin{proof}
	\ref{it-lo1}$\Rightarrow$\ref{it-lo2} Suppose that $f$ is path-slice. Then there is a path-slice stem function $F:\mathscr{P}(\mathbb{C}^d,\Omega)\rightarrow (\mathbb{R}^{2n})^{2\times 1}$ of $f$. Let $t\in[0,1]$, $\gamma\in\mathscr{P}(\mathbb{C}^d,\Omega)$, $J=(J_1,...,J_k)\in\left[\mathcal{C}(\gamma,\Omega)\right]^k$ and $I\in\mathcal{C}(\gamma,\Omega,J)$. It is easy to check that
	\begin{equation*}
		J\in\left[\mathcal{C}(\gamma[t],\Omega)\right]^k\qquad\mbox{and}\qquad I\in\mathcal{C}\left(\gamma[t],\Omega,J\right)\subset\mathcal{C}_{ker}(J).
	\end{equation*}
	Since $f$ is path-slice, we have
	\begin{equation*}
		f\circ\gamma^{K}(t)=f\circ\gamma[t]^{K}(1)=(1,K)F(\gamma[t]),\qquad K=I,J_1,...,J_k.
	\end{equation*}
	Hence
	\begin{equation*}
		\begin{cases}
			f\circ\gamma^I(t)=(1,I)F(\gamma[t]),
			\\f\circ\gamma^J(t)=\zeta(J)F(\gamma[t]).
		\end{cases}
	\end{equation*}
	According to $I\in\mathcal{C}_{ker}(J)$ and \eqref{eq-ij},
	\begin{equation*}
		f\circ\gamma^I(t)=(1,I)F(\gamma[t])=(1,I)\zeta^+(J)\zeta(J)F(\gamma[t])=(1,I)\zeta^+(J)f\circ\gamma^J(t).
	\end{equation*}
	It is clear that $f$ is path-pseudoslice, since the choice of $t$ is arbitrary.
	
	\ref{it-lo2}$\Rightarrow$\ref{it-lo3} Suppose that $f$ is path-pseudoslice. Let $\gamma\in\mathscr{P}(\mathbb{C}^d,\Omega)$ and $J\in\left[\mathcal{C}(\gamma,\Omega)\right]^k$ be a slice-solution of $\mathcal{C}(\gamma,\Omega)$. Then $\mathcal{C}(\gamma,\Omega)=\mathcal{C}(\gamma,\Omega,J)$. Since $f$ is path-pseudoslice, by definition it follows that for each $I\in\mathcal{C}(\gamma,\Omega)=\mathcal{C}(\gamma,\Omega,J)$,
	\begin{equation*}
		f\circ\gamma^I=(1,I)\zeta^+(J)(f\circ\gamma^J).
	\end{equation*}
	Therefore $f$ is path-solution slice.
	
	\ref{it-lo3}$\Rightarrow$\ref{it-lo1} Suppose that $f$ is path-solution slice. By Proposition \ref{pr-lss}, for each $\gamma\in\mathscr{P}(\mathbb{C}^d,\Omega)$, we can choose a slice-solution $J^{\gamma}\in\left[\mathcal{C}(\gamma,\Omega)\right]^k$ of $\mathcal{C}(\gamma,\Omega)$. Define a function $F:\mathscr{P}(\mathbb{C}^d,\Omega)\rightarrow (\mathbb{R}^{2n})^{2\times 1}$ by
	\begin{equation*}
		F(\gamma):=\zeta^+(J^\gamma)(f\circ\gamma^{(J^\gamma)}).
	\end{equation*}
	It is clear by \eqref{eq-fcgi} that $F$ is a path-slice stem function of $f$ and $f$ is path-slice.
\end{proof}

\begin{cor}\label{co-ews}
	Each weak slice regular function defined on a slice-open set in $\mathcal{W}_\mathcal{C}^d$ is path-slice and path-solution slice.
\end{cor}

\begin{proof}
	This Corollary follows directly from Theorem \ref{th-oi} and Proposition \ref{pr-fsa}.
\end{proof}

\begin{cor}\label{co-fcg}
	(Another form of path-representation Formula) Let $\Omega\subset\mathcal{W}_\mathcal{C}^d$, $\gamma\in\mathscr{P}(\mathbb{C}^d,\Omega)$, and $J\in\left[\mathcal{C}(\gamma,\Omega)\right]^k$ be a slice-solution of $\mathcal{C}(\gamma,\Omega)$. If $f:\Omega\rightarrow\mathbb{R}^{2n}$ is weak slice regular, then for each $I\in\mathcal{C}(\gamma,\Omega)$,
	\begin{equation}\label{eq-fgg}
		f\circ\gamma^I=(1,I)\zeta^+(J)(f\circ\gamma^J),
	\end{equation}
	where $f\circ\gamma^J$ is defined by \eqref{eq-fgj}.
\end{cor}

\begin{proof}
	This corollary follows directly from Corollary \ref{co-ews}.
\end{proof}

Let $J_1,J_2\in\mathcal{C}$. Note that
\begin{equation*}
	J_1(J_1-J_2)=-(J_1-J_2)J_2\ \left(=-1-J_1J_2\right).
\end{equation*}
If $J_1-J_2$ is invertible, then
\begin{equation*}
	(J_1-J_2)^{-1}J_1=-J_2(J_1-J_2)^{-1}.
\end{equation*}
One can easily verify that
\begin{equation}\label{eq-1ji}
	\begin{pmatrix}
		1&J_1\\1&J_2
	\end{pmatrix}^{-1}
	=\begin{pmatrix}
		(J_1-J_2)^{-1}J_1&-(J_1-J_2)^{-1}J_2\\(J_1-J_2)^{-1}&-(J_1-J_2)^{-1}
	\end{pmatrix}.
\end{equation}

\begin{cor} (Classical path-representation Formula)
	Let $\Omega\subset\mathcal{W}_\mathcal{C}^d$, $f:\Omega\rightarrow\mathbb{R}^{2n}$ be a weak slice regular function, $\gamma\in\mathscr{P}(\mathbb{C}^d,\Omega)$ and $I,J_1,J_2\in\mathcal{C}(\gamma,\Omega)$ with $J_1-J_2$ being invertible. Then
	\begin{equation}\label{eq-fcip}
		f\circ\gamma^I=(1,I)
		\begin{pmatrix}
			1&J_1\\1&J_2
		\end{pmatrix}^{-1}
		\begin{pmatrix}
			f\circ\gamma^{J_1}\\f\circ\gamma^{J_2}
		\end{pmatrix}.
	\end{equation}
\end{cor}

\begin{proof}
	Set $J=(J_1,J_2)$. Since $J_1-J_2$ is invertible, it follows from \eqref{eq-1ji} that
	\begin{equation*}
		\zeta(J)=\begin{pmatrix}
			1&J_1\\1&J_2
		\end{pmatrix}
	\end{equation*}
	is invertible, so is $\Im^{-1}_{D_J}\zeta(J)$.
	Set $J=(J_1,J_2)$. By Proposition \ref{pr-ekm} \ref{it-ekm1},
	\begin{equation*}
		\left[\Im^{-1}_{D_J}\zeta(J)\right]^+=\left[\Im^{-1}_{D_J}\zeta(J)\right]^{-1}=\zeta(J)^{-1}\left(\Im^{-1}_{D_J}\right)^{-1}.
	\end{equation*}
	Hence
	\begin{equation}\label{eq-zjl}
		\zeta^+(J)=\left[\Im^{-1}_{D_J}\zeta(J)\right]^+\Im^{-1}_{D_J}=\zeta(J)^{-1}\left(\Im^{-1}_{D_J}\right)^{-1}\Im^{-1}_{D_J}=\zeta(J)^{-1}.
	\end{equation}
	On the other hand, by Example \ref{ex-ij}, $J$ is a slice-solution of $\mathcal{C}(\gamma,\Omega)$. Therefore equality \eqref{eq-fcip} follows directly from  \eqref{eq-fgg} and \eqref{eq-zjl}.
\end{proof}

\section{Weak slice regular functions over LSCS algebras}\label{sc-ws}

Despite what happens for the various types of slice regular functions treated in the literature, the weak slice regular functions we considered in the previous sections have values in a real vector space, not in an algebra, and their domain $\mathcal{W}_{\mathcal{C}}^d$ is not a subset (proper or improper) of the codomain. In order to consider functions with values in a real algebra it is necessary to require additional hypothesis on the algebra at hand. We call such algebras {\em left slice complex structure algebras}, or LSCS algebras for short.

The weak slice analysis for LSCS algebras includes the slice analysis already known and studied in the literature in the case of quaternions, Clifford algebras, octonions, but it also provides new  frameworks, for example the case of left alternative algebras and of sedenions.

\subsection{LSCS algebras}

In this subsection, we define the so-called LSCS algebras. We choose a cone $\mathcal{W}_A^d$ for the LSCS algebra $A$ and define the slice-topology on it. We state some definitions and propositions without proofs, since they are similar to those in Section \ref{sc-wsc}.

Let $A$ be a real algebra, whose binary multiplication operation is not assumed to be associative and recall that  $L_a:\ A\rightarrow A$ denotes the multiplication on the left by $a\in A$.
\begin{defn}
	A  finite-dimensional real, unital algebra $A\neq\{0\}$ is called a left slice complex structure algebra, LSCS algebra for short, if there is $a\in A$ such that $L_a$ is a complex structure on $A$.
\end{defn}

\begin{exa}
	Complex numbers, quaternions, octonions, Clifford algebras $\mathbb R_{n,m}$ ($(n,m)\neq(1,0)$) and real alternative $*$-algebras are examples of LSCS algebras. Moreover, also sedenions or other Cayley–Dickson algebras are LSCS algebras.
\end{exa}

Assume that $A$ is an LSCS algebra. Since there is a complex structure on $A$, the real dimension of $A$ is even. Hence we can set
\begin{equation*}
	n:=\frac{1}{2}\dim (A),
\end{equation*}
so that $A\cong\mathbb{R}^{2n}$ as a real vector space.

\begin{defn}
	We call
	\begin{equation*}
		\mathcal{W}_A^d:=\bigcup_{I\in\mathcal{S}_A}\mathbb{C}_I^d
	\end{equation*}
	the (left) weak-cone of $A$. Let
	\begin{equation}\label{eq-ca}
		\mathcal{S}_A:=\left\{a\in A:L_a^2=-id_A\right\}
	\end{equation}
	the so-called the set of (left) slice-units of $A$.
\end{defn}
It is immediate that
\begin{equation*}
	L(\mathcal{S}_A):=\{L_I\ :\ I\in\mathcal{S}_A\}
\end{equation*}
is a symmetric subset of $\mathfrak{C}(A)$ $\left(\cong\mathfrak{C}\left(\mathbb{R}^{2n}\right)\right)$, and
\begin{equation*}
	L[d]\left(\mathcal{W}_A^d\right)=\{(L_{q_1},...,L_{q_d}):(q_1,...,q_d)\in\mathcal{W}_A^d\}
\end{equation*}
is a weak slice-cone of
$A$ identified with $\mathbb{R}^{2n}$.

\begin{lem}
The set
	\begin{equation}\label{eq-tlm}
		\tau_s\left(\mathcal{W}_A^d\right):=\{\Omega\subset \mathcal{W}_A^d:\Omega_I\in\tau(\mathbb{C}^d_I),\ \forall\ I\in\mathcal{S}_A\}.
	\end{equation}
	is a topology on $\mathcal{W}_{\mathcal{C}_A}^d$ (called the slice-topology on $\mathcal{W}_{A}^d$).
\end{lem}
The above lemma can be proved following the proof of Lemma \ref{lm-tlm}.
In reality, since $A$ has a real vector space structure with even dimension, all the statements in Section \ref{sc-wsc} can be repeated  for $\mathcal{W}_A^d$ with the same proofs.

\subsection{Weak slice regular functions over LSCS algebras}

In this subsection, we define weak slice regular functions on slice-open sets in $\mathcal{W}_A^d$, i.e. sets belonging to $\tau_s\left(\mathcal{W}_{A}^d\right)$, and with values in an LSCS algebra.
Properties in Section \ref{sc-wsr}, \ref{sc-pf} and \ref{sc-prf} also hold for the weak slice regular functions over LSCS algebras,
including a splitting lemma, an identity principle and a representation formula.
   We state some lemmas and theorems without proofs, since they are similar to the corresponding ones in Section \ref{sc-wsr}, \ref{sc-pf} and \ref{sc-prf}.

\begin{defn}
	Let $\Omega\in\tau_s(\mathcal{W}_{A}^d)$. A function $f:\Omega\rightarrow A$ is called weak slice regular if and only if for each $I\in\mathcal{S}_A$, $f_I=f|_{\Omega_I}$ is real differentiable and for any $\ell=1,2,...,d$,
	\begin{equation}\label{eq-wff}
		\frac{1}{2}\left(\frac{\partial}{\partial x_\ell}+I\frac{\partial}{\partial y_\ell}\right)f_I(x+yI)=0,\qquad\mbox{on}\qquad\Omega_I.
	\end{equation}
\end{defn}

It is evident that the monomial
	\begin{equation*}
		(x+y I)\shortmid\!\longrightarrow (x+yI)^{\alpha} a=(x_1+y_1I)^{\alpha_1}\cdots (x_d+y_dI)^{\alpha_d} a,
	\end{equation*}
	is a weak slice regular function, where $x,y\in\mathbb{R}^d$, $a\in A$, $\alpha=(\alpha_1,\ldots, \alpha_d)\in\mathbb{N}^d$ and $I\in\mathcal{S}_A$.

Let $\Omega\in\tau_s(\mathcal{W}_A^d)$. Note that for each weak slice regular function $f:\Omega\rightarrow A$, equality \eqref{eq-wff} holds for each $I\in\mathcal{S}_A$. It implies that
\begin{equation*}
	\begin{split}
		&\frac{1}{2}(\frac{\partial}{\partial x_\ell}+L_I\frac{\partial}{\partial y_\ell})f\circ L[d]^{-1}|_{L[d](\Omega)}
		(x+yL_I)
		\\=&\frac{1}{2}(\frac{\partial}{\partial x_\ell}+I\frac{\partial}{\partial y_\ell})f_I(x+yI)=0,
	\end{split}
\end{equation*}
for any $L_I\in L(\mathcal{S}_A)$, $x+yL_I\in L(\Omega)$ and $\ell=1,2,...,d$. It is clear that
\begin{equation*}
	f\circ
	L[d]^{-1}|_{L[d](\Omega)}:L[d](\Omega)
	\longrightarrow A\cong\mathbb{R}^{2n}
\end{equation*}
is a $\mathbb{R}^{2n}$-valued weak slice regular function over the weak slice-cone $\mathcal{W}_{L(\mathcal{S}_A)}^d$ of $\mathbb{R}^{2n}$, where $L[d](\Omega)\in\tau_s\left(\mathcal{W}_{L(\mathcal{S}_A)}^d\right)$.

Similarly, if $f\circ L[d]^{-1}|_{L[d](\Omega)}$ is a weak slice regular function, $f$ is also a weak slice regular function. In summary, the following proposition holds.

\begin{prop}\label{pr-wlo1}
	Let $\Omega\in\tau_s\left(\mathcal{W}_A^d\right)$. A function $f:\Omega\rightarrow A$ is weak slice regular if and only if $f\circ L[d]^{-1}|_{L[d](\Omega)}$ is weak slice regular.
\end{prop}

Proposition \ref{pr-wlo1} implies that various properties of weak slice regular functions over weak slice-cones of $\mathbb{R}^{2n}$ also hold for weak slice regular functions over $A$, among which the Splitting Lemma and the Identity Principle that we do not repeat here. We only adapt the notations and the statement of the Path-representation Formula.
%
%

Let $\Omega\subset\mathcal{W}_{A}^d$ and $\gamma\in\mathscr{P}\left(\mathbb{C}^d\right)$. Define
\begin{equation*}
	\mathscr{P}\left(\mathbb{C}^d,\Omega\right):=\left\{\delta\in\mathscr{P}\left(\mathbb{C}^d\right):\exists\ I\in\mathcal{S}_A,\mbox{ s.t. }Ran(\delta^{I})\subset\Omega\right\}
\end{equation*}
and
\begin{equation*}
	\mathcal{S}_A(\gamma,\Omega):=\left\{I\in\mathcal{S}_A:Ran(\gamma^{I})\subset\Omega\right\}.
\end{equation*}

Let $J=(J_1,...,J_k)\in\mathcal{S}_A^k$. Define
\begin{equation*}
	\mathcal{S}_A^{ker}(J):=\left\{I\in\mathcal{S}_A:\ker(1,L_I)\supset\bigcap_{\ell=1}^k\ker(1,L_J)\right\}.
\end{equation*}
Let $\Omega\subset\mathcal{W}_\mathcal{C}^d$ and $\gamma\in\mathscr{P}(\mathbb{C}^d,\Omega)$. Set
\begin{equation*}
	\mathcal{S}_A(\Omega,\gamma,J):=\mathcal{S}_A(\Omega,\gamma)\cap\mathcal{S}_A^{ker}(J).
\end{equation*}

\begin{thm}
	(Path-representation Formula) Let $\Omega$ be a slice-open set in $\mathcal{W}_{A}^d$, $\gamma\in\mathscr{P}(\mathbb{C}^d,\Omega)$, $J=(J_1,J_2,...,J_k)\in\left[\mathcal{S}_A(\gamma,\Omega)\right]^k$ and $I\in\mathcal{S}_A(\gamma,\Omega,J)$. If $f:\Omega\rightarrow A$ is weak slice regular then
	\begin{equation*}
		f\circ\gamma^I=(1,L_I)\zeta^+(L_J)(f\circ\gamma^J),
	\end{equation*}
	where $f\circ\gamma^J$ is defined by \eqref{eq-fgj}.
\end{thm}

\begin{cor}\label{correprform}
	Let $\Omega\subset\mathcal{W}_{A}^d$, $f:\Omega\rightarrow A$ be a weak slice regular function, $\gamma\in\mathscr{P}(\mathbb{C}^d,\Omega)$ and $I,J_1,J_2\in\mathcal{S}_A(\gamma,\Omega)$ with $L_{J_1}-L_{J_2}$ being invertible. Then
	\begin{equation}\label{eq-fee}
		f\circ\gamma^I=(1,L_I)
		\begin{pmatrix}
			1&L_{J_1}\\1&L_{J_2}
		\end{pmatrix}^{-1}
		\begin{pmatrix}
			f\circ\gamma^{J_1}\\f\circ\gamma^{J_2}
		\end{pmatrix},
	\end{equation}
	where
	\begin{equation}\label{eq-ljj}
		\begin{pmatrix}
			1&L_{J_1}\\1&L_{J_2}
		\end{pmatrix}^{-1}
		=\begin{pmatrix}
			(L_{J_1}-L_{J_2})^{-1}L_{J_1}&-(L_{J_1}-L_{J_2})^{-1}L_{J_2}\\(L_{J_1}-L_{J_2})^{-1}&-(L_{J_1}-L_{J_2})^{-1}
		\end{pmatrix}.
	\end{equation}
\end{cor}

The above corollary is the classical path-representation formula. Now we write this formula in a classical form. By \eqref{eq-fee} and \eqref{eq-ljj}, we have
\begin{equation}\label{reprfor}
	\begin{split}
		f\circ\gamma^I
		=(I-J_2)\left[(J_1-J_2)^{-1}f\circ\gamma^{J_1}\right]
		-(I-J_1)\left[(J_1-J_2)^{-1}f\circ\gamma^{J_2}\right],
	\end{split}
\end{equation}
which is in the form of the representation formula for slice function over real alternative $*$-algebra in \cite[Proposition 6]{Ghiloni2011001}.

\section{The case of left alternative algebras}\label{sc-laa}
In \cite{Ghiloni2011001}, Ghiloni and Perotti define a notion of slice regular functions in one variable for functions with values in a real alternative $*$-algebra and using stem functions. This class does not coincide with the one considered in this work and in fact in this section we show that the class of weak slice regular functions on a finite-dimensional real left alternative algebra $A$ is a generalization of the class of slice regular functions defined in \cite{Ghiloni2011001}.

We recall that a real algebra $A$ is called left alternative, if for any $a,b\in A$,
\begin{equation*}
	a(ab)=(aa)b.
\end{equation*}
The interested reader may consult  \cite{Albert1949001} in which the author introduces the notion of right alternative algebra that can be easily adapted to our case.

\begin{ass}\label{as-wa}
	Assume that $A\neq\{0\}$ is a finite-dimensional real unital left alternative algebra with $\mathcal{S}_A\neq\varnothing$.
\end{ass}

We extend the quadratic cone in \cite{Ghiloni2011001} to a real unital left alternative algebra $A$ by
\begin{equation*}
	Q_A:=\bigcup_{I\in\mathbb{S}_A}\mathbb{C}_I,
\end{equation*}
where $\mathbb{S}_A$ is the set of imaginary units, defined by
\begin{equation*}
	\mathbb{S}_A:=\{a\in A:a^2=-1\}.
\end{equation*}

Since for every $I\in\mathbb{S}_A$,
\begin{equation*}
	(L_I)^2(a)=I(Ia)=(II)a=-a,
\end{equation*}
it follows that $L_I$ is a complex structure on $A$. Hence
\begin{equation}\label{eq-cam}
	\mathbb{S}_A=\mathcal{S}_A,\qquad\mbox{and}\qquad Q_A=\mathcal{W}_A^1.
\end{equation}

\subsection{Real alternative $*$-algebras, case $d=1$}

The results in Section \ref{sc-ws} hold in the particular case of real alternative $*$-algebras with unit and finite dimensional (this is the case considered in \cite{Ghiloni2011001}) and  as an example we shall consider the case $d=1$.

By \eqref{eq-tlm} and \eqref{eq-cam}, the slice-topology on $Q_A(=\mathcal{W}_A^1)$ is
	\begin{equation*}
		\tau_s(Q_A):=\{\Omega\subset Q_A:\Omega_I\in\tau(\mathbb{C}_I),\ \forall\ I\in\mathbb{S}_A\}.
	\end{equation*}

\begin{defn}\label{sl-due}
	Let $\Omega\in\tau_s(Q_A)$. A function $f:\Omega\rightarrow A$ is called weak slice regular if and only if for each $I\in\mathbb{S}_A$, $f_I:=f|_{\Omega_I}$ is real differentiable and
	\begin{equation*}
		\frac{1}{2}\left(\frac{\partial}{\partial x}+I\frac{\partial}{\partial y}\right)f_I(x+yI)=0,\qquad\mbox{on}\qquad\Omega_I.
	\end{equation*}
\end{defn}

\begin{thm}
	(Identity Principle) Let $\Omega$ be a domain in $\tau_s(Q_{A})$ and $f,g:\Omega\rightarrow A$ be weak slice regular. If $f=g$ on a non-empty open subset of $\Omega_I$ (or $\Omega_\mathbb{R})$ for some $I\in\mathbb{S}_A$, then $f=g$ on $\Omega$.
\end{thm}

Let $U$ be an open set in $\mathbb{C}$ and
\begin{equation*}
	\Omega_U:=\{x+yI\in Q_A:x+yi\in U,\ I\in\mathbb{S}_A\}.
\end{equation*}
\begin{rmk}
For any open subset $U$ in $\mathbb C$ the set $\Omega_U$ is slice-open in $Q_A$. Suppose that $U$ is symmetric with respect to the real axis and that $f:\Omega_U\rightarrow A$ is a slice regular function according to \cite[Definition 8]{Ghiloni2011001}. By \cite[Proposition 8]{Ghiloni2011001}, we have that $f_I$ is holomorphic for all $I\in\mathbb{S}_A$. Thus these slice regular functions are weak slice regular.

On the other hand, all slice regular functions according to \cite[Definition 8]{Ghiloni2011001} are slice functions. However, not all the weak slice regular functions are of slice type. Hence weak slice regular functions form a class  larger than the one defined in \cite{Ghiloni2011001}.
\end{rmk}
An important result that we obtained in Section 6 and that holds in this case is:
\begin{prop}
	(Classical Path-representation Formula) Let $\Omega\subset Q_A$, $f:\Omega\rightarrow A$ be a weak slice regular function, $\gamma\in\mathscr{P}(\mathbb{C}^1,\Omega)$ and $I,J_1,J_2\in \mathbb{S}_A(\gamma,\Omega)$ with $L_{J_1}-L_{J_2}$ being invertible. Then
	\begin{equation*}
		f\circ\gamma^I=(1,L_I)
		\begin{pmatrix}
			1&L_{J_1}\\1&L_{J_2}
		\end{pmatrix}^{-1}
		\begin{pmatrix}
			f\circ\gamma^{J_1}\\f\circ\gamma^{J_2}
		\end{pmatrix},
	\end{equation*}
	where $\begin{pmatrix}
		1&L_{J_1}\\1&L_{J_2}
	\end{pmatrix}^{-1}$ satisfies \eqref{eq-ljj} and $\mathbb{S}_A(\gamma,\Omega):=\{I\in\mathbb{S}_A:Ran(\gamma^I)\subset\Omega\}$.
\end{prop}

\subsection{Clifford algebra $\mathbb{R}_{m}$ case for $d=1$}

Slice regular functions with values in the real Clifford algebra $\mathbb{R}_{0,m}=\mathbb R_m$ are called slice monogenic functions and were introduced in \cite{Colombo2009002}.

Let $\mathbb{R}_m$ be the real Clifford algebra over $m$ units $\mathbf{e}_1,...,\mathbf{e}_m$ such that $\mathbf{e}_\ell \mathbf{e}_\jmath+\mathbf{e}_\jmath \mathbf{e}_\ell=-2\delta_{\ell,\jmath}$, where $\delta_{\ell\jmath}$ is the Kronecker symbol.

In \cite{Colombo2009002}, the slice monogenic functions are defined on the real vector space
\begin{equation*}
	\mathbf{R}^{m+1}:=\bigcup_{I\in\mathbb{S}_{\mathbb{R}_m}^*}\mathbb{C}_I,
\end{equation*}
in $\mathbb{R}_m$, where
\begin{equation*}
	\mathbb{S}_{\mathbb{R}_m}^*:=\{x_1\mathbf{e}_1+\cdots x_m \mathbf{e}_m:x_1^2+\cdots+x_m^2=1\}.
\end{equation*}
The cone $\mathbf{R}^{m+1}$ is a $m+1$-dimensional real vector space, i.e.
\begin{equation*}
	\mathbf{R}^{m+1}=\mathbb{R}+\mathbf{e}_1\mathbb{R}+...+\mathbf{e}_m\mathbb{R}\cong\mathbb{R}^{m+1}
\end{equation*}
which can be identified with the set of paravectors in $\mathbb R_m$.
In \cite{Ghiloni2011001}, slice regular functions are generalized to a larger set, namely the quadratic cone of $\mathbb{R}_m$
\begin{equation*}
	Q_{\mathbb{R}_m}:=\bigcup_{I\in\mathbb{S}_{\mathbb{R}_m}}\mathbb{C}_I,
\end{equation*}
where
\begin{equation*}
	\mathbb{S}_{\mathbb{R}_m}:=\{I\in\mathbb{R}_m:I^2=-1\}.
\end{equation*}
With the approach in this paper we can consider this second case and, by \eqref{eq-tlm}, we can introduce  the slice-topology on $Q_{\mathbb{R}_m}=\mathcal{W}_{\mathbb{R}_m}$:
\begin{equation*}
	\tau_s(Q_{\mathbb{R}_m})=\{\Omega\subset Q_{\mathbb{R}_m}:\Omega_I\in\tau(\mathbb{C}_I),\ \forall\ I\in\mathbb{S}_{\mathbb{R}_m}\}.
\end{equation*}
Weak slice regularity according to Definition \ref{sl-due} can be considered for functions
$f:\Omega\rightarrow {\mathbb{R}_m}$, where $\Omega\in\tau_s(Q_{\mathbb{R}_m})$.
All the results proved in sections 4-7 hold in this case and in particular Corollary \ref{correprform} which leads to the classical representation formula \ref{reprfor}.

\section{The case of sedenions}\label{sc-sc}

In this section we consider the algebra of sedenions.
The sedenions $\mathfrak{S}$ are obtained in \cite{Muses1980001} by applying the Cayley–Dickson construction \cite{Dickson1919001} to the octonions. The case of functions with values in this algebra is completely new in slice analysis since sedenions are not an alternative algebra.

Elements in the algebra of sedenions $\mathfrak{S}$ are of the form
\begin{equation*}
	s=\sum_{m=0}^{15}x_m e_m,
\end{equation*}
where $x_m$ are reals, $e_0=1$ and $e_1,e_2,...,e_{15}$ are imaginary units (i.e. their square equals $-1$) and satisfy the multiplication table in the Appendix.

The set of slice-units $\mathcal{S}_\mathfrak{S}$ in $\mathfrak{S}$ is defined as we already did in other cases by
\begin{equation*}
	\mathcal{S}_\mathfrak{S}:=\left\{s\in\mathfrak{S}:L_s^2=-id_{\mathfrak{S}}\right\}.
\end{equation*}
Let us denote by $\mathbb O$ the algebra of octonions which is isomorphic to the algebra generated by $e_1,\ldots, e_7$, i.e.
	\begin{equation*}
		\mathbb{O}:=\mathbb{R}\langle1,e_1,...,e_7\rangle .
	\end{equation*}
The set
	\begin{equation*}
		\mathbb{S}_{\mathfrak{S}}:=\{s\in\mathfrak{S}:s^2=-1\}.
	\end{equation*}
	contains all the imaginary units in $\mathfrak{S}$.
	
By direct calculations, one may show that
	\begin{equation*}
		\mathcal{S}_{\mathfrak{S}}=\left\{p+qe_8\in\mathbb{S}_{\mathfrak{S}}:p,q\in\mathbb{O},\ pq=qp\right\},
	\end{equation*}
	
It is clear that $\mathfrak{S}$ is an LSCS algebra. The weak slice regular functions can be defined on slice-open sets in
\begin{equation*}
	\mathcal{W}_\mathfrak{S}:=\bigcup_{I\in\mathcal{S}_\mathfrak{S}}\mathbb{C}_I.
\end{equation*}

By \eqref{eq-tlm} and \eqref{eq-cam}, the slice-topology on $\mathcal{W}_\mathfrak{S}(=\mathcal{W}_\mathfrak{S}^1)$ is
\begin{equation*}
	\tau_s(\mathcal{W}_\mathfrak{S}):=\{\Omega\subset \mathcal{W}_\mathfrak{S}\ :\Omega_I\in\tau(\mathbb{C}_I),\ \forall\ I\in\mathcal{S}_\mathfrak{S}\}.
\end{equation*}
The definition of weak slice regular function can be repeated as in the previous sections:
\begin{defn}
	Let $\Omega\in\tau_s(\mathcal{W}_\mathfrak{S})$. A function $f:\Omega\rightarrow \mathfrak{S}$ is called weak slice regular if and only if for each $I\in\mathcal{C}_\mathfrak{S}$, $f_I:=f|_{\Omega_I}$ is real differentiable and
	\begin{equation*}
		\frac{1}{2}\left(\frac{\partial}{\partial x}+I\frac{\partial}{\partial y}\right)f_I(x+yI)=0,\qquad\mbox{on}\qquad\Omega_I.
	\end{equation*}
\end{defn}
All the relevant results hold in the case of sedenions and, in particular, the following proposition deserves a comment:
\begin{prop}
  	(Classical Path-representation Formula) Let $\Omega\subset\mathcal{W}_\mathfrak{S}$, $f:\Omega\rightarrow\mathfrak{S}$ be a weak slice regular function, $\gamma\in\mathscr{P}(\mathbb{C}^1,\Omega)$ and $I,J_1,J_2\in\mathcal{S}_\mathfrak{S}(\gamma,\Omega)$ with $L_{J_1}-L_{J_2}$ being invertible. Then
	\begin{equation*}
		f\circ\gamma^I=(1,L_I)
		\begin{pmatrix}
			1&L_{J_1}\\1&L_{J_2}
		\end{pmatrix}^{-1}
		\begin{pmatrix}
			f\circ\gamma^{J_1}\\f\circ\gamma^{J_2}
		\end{pmatrix},
	\end{equation*}
	where $\begin{pmatrix}
		1&L_{J_1}\\1&L_{J_2}
	\end{pmatrix}^{-1}$ satisfies \eqref{eq-ljj}.
\end{prop}
In \eqref{eq-ljj} we need to compute $L_{J_1}-L_{J_2}$ which does not always exist. For example, let $J_1=e_1\in\mathcal{S}_\mathfrak{S}$ and $J_2=-e_{10}\in\mathcal{S}_\mathfrak{S}$. Since
\begin{equation*}
	(e_1+e_{10})(e_5+e_{14})=0,
\end{equation*}
it is clear that $L_{J_1}-L_{J_2}=L_{e_1+e_{10}}$ is not invertible. Then we can write the following path-representation formula in which we use  the Moore-Penrose inverse.

\begin{thm}
	(Path-representation Formula) Let $\Omega$ be a slice-open set in $\mathcal{W}_{\mathfrak{S}}$, $\gamma\in\mathscr{P}(\mathbb{C}^d,\Omega)$, $J=(J_1,J_2,...,J_k)\in\left[\mathcal{S}_\mathfrak{S}(\gamma,\Omega)\right]^k$ and $I\in\mathcal{S}_\mathfrak{S}(\gamma,\Omega,J)$. If $f:\Omega\rightarrow \mathfrak{S}$ is weak slice regular, then
	\begin{equation*}
		f\circ\gamma^I=(1,L_I)\zeta^+(L_J)(f\circ\gamma^J),
	\end{equation*}
	where $f\circ\gamma^J$ is defined by \eqref{eq-fgj}.
\end{thm}

\bibliographystyle{amsplain}
\bibliography{mybibfile}
\newpage

\section*{Appendix: Sedenion multiplication table}

\begin{table}[h]
	\resizebox{\textwidth}{3.5cm}{
		\begin{tabular}{|c"c c c c|c c c c|c c c c|c c c c|}
			\thickhline$\cdot$&$1$&$e_1$&$e_2$&$e_3$&$e_4$&$e_5$&$e_6$&$e_7$&$e_8$&$e_9$&$e_{10}$&$e_{11}$&$e_{12}$&$e_{13}$&$e_{14}$&$e_{15}$\\
			\thickhline$1$&$1$&$e_1$&$e_2$&$e_3$&$e_4$&$e_5$&$e_6$&$e_7$&$e_8$&$e_9$&$e_{10}$&$e_{11}$&$e_{12}$&$e_{13}$&$e_{14}$&$e_{15}$\\
			$e_1$&$e_1$&$-1$&$e_3$&$-e_2$&$e_5$&$-e_4$&$-e_7$&$e_6$&$e_9$&$-e_8$&$-e_{11}$&$e_{10}$&$-e_{13}$&$e_{12}$&$e_{15}$&$-e_{14}$\\
			$e_2$&$e_2$&$-e_3$&$-1$&$e_1$&$e_6$&$e_7$&$-e_4$&$-e_5$&$e_{10}$&$e_{11}$&$-e_8$&$-e_9$&$-e_{14}$&$-e_{15}$&$e_{12}$&$e_{13}$\\
			$e_3$&$e_3$&$e_2$&$-e_1$&$-1$&$e_7$&$-e_6$&$e_5$&$-e_4$&$e_{11}$&$-e_{10}$&$e_9$&$-e_8$&$-e_{15}$&$e_{14}$&$-e_{13}$&$e_{12}$\\
			\hline$e_4$&$e_4$&$-e_5$&$-e_6$&$-e_7$&$-1$&$e_1$&$e_2$&$e_3$&$e_{12}$&$e_{13}$&$e_{14}$&$e_{15}$&$-e_8$&$-e_9$&$-e_{10}$&$-e_{11}$\\
			$e_5$&$e_5$&$e_4$&$-e_7$&$e_6$&$-e_1$&$-1$&$-e_3$&$e_2$&$e_{13}$&$-e_{12}$&$e_{15}$&$-e_{14}$&$e_9$&$-e_8$&$e_{11}$&$-e_{10}$\\
			$e_6$&$e_6$&$e_7$&$e_4$&$-e_5$&$-e_2$&$e_3$&$-1$&$-e_1$&$e_{14}$&$-e_{15}$&$-e_{12}$&$e_{13}$&$e_{10}$&$-e_{11}$&$-e_8$&$e_9$\\
			$e_7$&$e_7$&$-e_6$&$e_5$&$e_4$&$-e_3$&$-e_2$&$e_1$&$-1$&$e_{15}$&$e_{14}$&$-e_{13}$&$-e_{12}$&$e_{11}$&$e_{10}$&$-e_9$&$-e_8$\\
			\hline$e_8$&$e_8$&$-e_9$&$-e_{10}$&$-e_{11}$&$-e_{12}$&$-e_{13}$&$-e_{14}$&$-e_{15}$&$-1$&$e_1$&$e_2$&$e_3$&$e_4$&$e_5$&$e_6$&$e_7$\\
			$e_9$&$e_9$&$e_8$&$-e_{11}$&$e_{10}$&$-e_{13}$&$e_{12}$&$e_{15}$&$-e_{14}$&$-e_1$&$-1$&$-e_3$&$e_2$&$-e_5$&$e_4$&$e_7$&$-e_6$\\
			$e_{10}$&$e_{10}$&$e_{11}$&$e_8$&$-e_9$&$-e_{14}$&$-e_{15}$&$e_{12}$&$e_{13}$&$-e_2$&$e_3$&$-1$&$-e_1$&$-e_6$&$-e_7$&$e_4$&$e_5$\\
			$e_{11}$&$e_{11}$&$-e_{10}$&$e_9$&$e_8$&$-e_{15}$&$e_{14}$&$-e_{13}$&$e_{12}$&$-e_3$&$-e_2$&$e_1$&$-1$&$-e_7$&$e_6$&$-e_5$&$e_4$\\
			\hline$e_{12}$&$e_{12}$&$e_{13}$&$e_{14}$&$e_{15}$&$e_8$&$-e_9$&$-e_{10}$&$-e_{11}$&$-e_4$&$e_5$&$e_6$&$e_7$&$-1$&$-e_1$&$-e_2$&$-e_3$\\
			$e_{13}$&$e_{13}$&$-e_{12}$&$e_{15}$&$-e_{14}$&$e_9$&$e_8$&$e_{11}$&$-e_{10}$&$-e_5$&$-e_4$&$e_7$&$-e_6$&$e_1$&$-1$&$e_3$&$-e_2$\\
			$e_{14}$&$e_{14}$&$-e_{15}$&$-e_{12}$&$e_{13}$&$e_{10}$&$-e_{11}$&$e_8$&$e_9$&$-e_6$&$-e_7$&$-e_4$&$e_5$&$e_2$&$-e_3$&$-1$&$e_1$\\
			$e_{15}$&$e_{15}$&$e_{14}$&$-e_{13}$&$-e_{12}$&$e_{11}$&$e_{10}$&$-e_9$&$e_8$&$-e_7$&$e_6$&$-e_5$&$-e_4$&$e_3$&$e_2$&$-e_1$&$-1$\\\thickhline
	\end{tabular}}
\end{table}

\end{document}